\documentclass[a4paper,12pt]{article}
\usepackage{amsxtra, amsbsy,amsgen,amsopn,  amsfonts, amssymb, amsthm}
\usepackage{mathrsfs}
    \newtheorem{thm}{Theorem}[section]%
    \newtheorem{prop}[thm]{Proposition}%
    \newtheorem{cor}[thm]{Corollary}%
    \newtheorem{lem}[thm]{Lemma}%
    \newtheorem{defn}[thm]{Definition}%
    \newtheorem{rem}[thm]{Remark}%

    \numberwithin{equation}{section}
\begin{document}
			\begin{center}
				{\bf  Cauchy problem in function spaces with asymptotic expansions with respect to
				 		time variable $t$ }   \\ \vspace{3mm}
				   					Sunao {\sc \={O}uchi}
				  \footnote{Sophia Univ. Tokyo Japan, e-mail s\_ouchi@sophia.ac.jp 
				 2020 Mathematical Classification Numbers. 35F55, Secondary 35A10, 41A60.}
			 \end{center}
	\begin{abstract}
		   A system of nonlinear Cauchy problem $\partial_t u_i=f_i(t,x, U, \nabla_xU )$  
			$u_i(0,x)= u_{i,0}(x)$ is studied in function spaces with asymptotic expansion 
			with respect to $t$. To be specific, it is discussed in Borel summable or 
			multisummable function space.
			It is recognized that these functions are important classes in asymptotic analysis. 
			We study equations under the condition $\{f_i(t,x, U, P)\}_{i=1}^m$ are in these function spaces 
			with respect to $t$ and show $\{u_i(t,x)\}_{i=1}^m$ have also the same summability. 
	\end{abstract}
	 \section{\normalsize Introduction}
	    	In this paper we study an initial value problem of a system of nonlinear partial differential 
		    equations in $(t,x)=(t,x_1, \dots, x_n)\in {\mathbb C}^{n+1}$, 
		\begin{equation} \left \{
			\begin{aligned} \
				 & \partial_t u_i=f_i\big(t,x, U, \nabla_xU \big) 
				 	\quad u_i(0,x)= u_{i,0}(x)\quad 1 \leq i \leq m, \\
				 & U=(u_1, \dots,u_m)\quad 
				   \nabla_x U =\{\partial_{x_j}u_i\}_{1\leq i \leq m, 1\leq j \leq n}, 
			\end{aligned} \right . \label{01}
		\end{equation}
			where the initial values $u_{i,0}(x)\;(1 \leq i \leq m)$ are holomorphic in a neighborhood of $x=0$.
			Let us denote variables corresponding to $\nabla_x u=\{\partial_{x_j}u_i\}$ by $P=(p_{i.j})$. 
			If $\{f_i\big(t,x, U,P\big)\}_{1\leq i \leq m}$	are holomorphic functions for all variables $
			(t,x,U,P)$, then well-known Cauchy-Kowalevskaja theorem asserts that 
			there exists a unique holomorphic solution represented by convergent power series   
			$	{u}_{i}(t,x)=\sum_{n=0}^{\infty}u_{i,n}(x)t^n.$ 
			The condition of analyticity of time $t$ can be weakened for \eqref{01}. Nagumo \cite{Nagumo} 
			 showed unique existence of solution 
			 under the condition $\{f_i\big(t,x, U, P \big)\}_{1\leq i \leq m}$ 
			are holomorphic for variables $(x,U,P)$ and continuous in $t$, that is, 			
			there exists a unique solution $U(t,x)$ which is holomorphic in $x$ and continuously differentiable 
			in $t$ in a domain $\{(t,x)\in {\mathbb R}\times{\mathbb C}^{n}; |t|<r_0, |x|<r\}$	
			(see also Nirenberg \cite{Nirenberg}).
			 \par
			Suppose that $\{f_i(t,x,U,P)\}_{i=1}^m$ are not necessarily holomorphic at $t=0$, but have asymptotic 
			expansion with respect to $t$ at $t=0$, 
			\begin{equation}
				f_{i}(t,x,u,p)\sim \sum_{n=0}^{\infty}f_{i,n}(x,U,P)t^n. \label{f's}
			\end{equation}
			Then there exists a unique formal power series solution 
			$\widetilde{U}(t,x)=(\widetilde{u}_{1},\cdots,	\widetilde{u}_{m})$ of \eqref{01},
			\begin{equation}
				\widetilde{u}_{i}(t,x)=\sum_{n=0}^{\infty}u_{i,n}(x)t^n,\; u_{i,0}(x)=u_{i}(x).\label{fs}
			\end{equation}
			The solution ${U}(t,x)$ assured in \cite{Nagumo} has asymptotic expansion \eqref{fs}. 
			We have a problem: 
			{\it What is the meaning of asymptotic expansion} \eqref{fs} ?\par
				 The main purpose of this paper is to study the meaning of \eqref{fs} more precisely. 
			In order to answer the above problem we use the theory of multisummable functions 
			in asymptotic analysis.  The set of Borel summable functions is its subclass, which is simpler and 
			can be treated more easily than general multisummabilty. 
			Multisummable and Borel summable functions are important notions
			(Balser \cite{Bal}), in particular asymptotic analysis. 
		  They give the width of the opening angle of a sector where asymptotic expansion is valid and the
			estimate of the error terms with $\Gamma$-function. They are firstly used to study ordinary 
			differential equations and give more precise  meaning of formal solutions than the preceding results 
			(Baler, Braaksam, Ramis, Sibuya \cite{BBRS}, Braaksam \cite{Bra}, Wasow \cite{Was}). \par
				They are also used to study partial differential equations. Firstly 
			Borel summability of formal power series solutions of initial problem of heat equation was 
			studied in Lutz, Miyake, Sch\"{a}fke \cite{LMS}. 
			It is suggested there that the situation for partial differential equations 
			is much different from that for ordinary differential equations. After \cite{LMS}
			 many researches follow (Hibino \cite{H}, Ichinobu, \cite{Ich}, 
				Malek \cite{Malek-1} \cite{Malek-2}, Luo, Chen, Zhang \cite{LCZ}, \={O}uchi \cite{O-2}, Remy  \cite{Re}
				and others) as for Borel summability. 
			They study Borel summability of solutions of partial differential equations under some conditions of 
			initial values or types of equations.
			 As for multisummability it is shown that solutions of some classes of partial differential 
			equations are	multisummable. Equations considered as perturbations of 
			ordinary differential equations are studied in \={O}uchi \cite{O-1},\cite{O-3} and linear equations with  
			depending only $t$ or constant coefficients are studied in Balser \cite{Bal-1},  Michalik 
			\cite{Mich}, Tahara Yamazawa \cite{TY}. \par
		    In this paper we answer for {\it "the meaning of asymptotic expansion \eqref{fs}"}. 
		\begin{enumerate}
			\item[(1)]	If $\{f_i(t,x,U,P)\}_{i=1}^m$ are Borel summable, 
				so are $\{\widetilde{u}_i(t,x)\}_{i=1}^{m}$ ?  
		  \item[(2)]
				If $\{f_i(t,x,U,P)\}_{i=1}^m$ are multisummable, so are $\{\widetilde{u}_i(t,x)\}_{i=1}^{m}$ ? 
		\end{enumerate}
			 We give affirmative answers of these problems in the present paper. \\
			The content is the following: 
			\begin{enumerate}
			  \item[\S 1] {\it Notations and Definitions}.
			   We give notations and definitions concerning formal series and Borel summable functions.
			 \item[\S 2-3]{\it Borel summability of solutions of Cauchy Problem}
			   	\begin{enumerate}
			   	  \item Preliminary.
					  \item Convolution equation. We change the initial value problem of partial differential
					    equations to a system of partial differential and  convolution equations. . 
					  \item Estimate and Convergence. It consists of 3 parts,  	
					    \begin{enumerate}	
					  \item Majorant functions. 
					  \item Estimate.
					  \item Convergence.
					  \end{enumerate}
			 		\end{enumerate}
			 	\item[\S 4-5]{\it Equations in Multisumamble	function spaces}. We give the definition of 
			 	${\bf k}$-multisumamble	function, ${\bf k}=(k_1,\dots,k_p)$, and its basic properties. 
			 	We study Cauchy problem in these function spaces. 
			 	Roughly speaking, we  repeat $p$ times the method used for Borel summable case.  
			 	\item[\S 6\;\;] Appendix.
				\end{enumerate}
		\section{\normalsize Notations and definitions}
				Firstly we introduce some notations and definitions. Let $\widetilde{{\mathbb C}}_{\{0\}}$ be the 
				universal covering space of  
				${\mathbb C}-\{0\}$ and $I=(a,b)$ be an open interval. Let us define some sectorial domains.
				$S(I):=\{t\in{\widetilde{{\mathbb C}}_{\{0\}}}; \arg t\in I\}$ is an infinite sector and 				 
				$S_{0}(I):=\{t\in \widetilde{{\mathbb C}}_{\{0\}}; \arg t \in I, 0<|t|<\rho(\arg t)\}$,
				where $\rho(t)$ is some positive continuous function on $I$. The same notation 
				$S_{0}(\cdot)$ is	used for various $\rho(\cdot)$. 
				 For arbitrary small $\epsilon>0$, set $I_{\epsilon}:=
				(a+\epsilon, b-\epsilon)\subset I$.  
		 		${\mathscr O}(\Omega)$ is the set of holomorphic functions on a domain $\Omega$. 
				${\mathbb N}$ is the set of nonnegative integers. 
				Let $\alpha=(\alpha_1,\dots,\alpha_n)\in  {\mathbb N}^n$ and 
		 		$Y=(y_1.\dots,y_n)\in {\mathbb C}^n$. 
		 		Then we use notations $\alpha!=\alpha_1!\alpha_2!\dots \alpha_n!$,
		 		$|\alpha|=\sum_{i=1}^{n}\alpha_i$, 
		 		$Y^{\alpha}=y_1^{\alpha_1}\dots y_n^{\alpha_n}$, $|Y|=\max_{1\leq i\leq n}|y_i|$ and	
		 		$(\frac{\partial}{\partial Y})^\alpha=\prod_{i=1}^{n}(\frac{\partial}{\partial y_i})^{\alpha_i}$.		
				${\mathscr O}(\Omega)[[t]]$ is the set of formal power series of $t$ with coefficients in
				${\mathscr O}(\Omega)$.
		    \begin{defn} Let $k>0$, $I=(a,b)$ with $b-a>\pi/k$ and 
				$\Omega=\{Y\in {\mathbb C}^n;	|Y|<R \}$. 			
					$\phi(t,Y)\in {\mathscr O}(S_{0}(I)\times \Omega)$ is said to be $k$-Borel summable 
					with respect to $t$, if  
					there exist constants $M$, $C$ and $\{a_n(y)\}_{n=0}^{\infty}\subset {\mathscr O}(\Omega)$ 
					such that for any $N\geq 0$
			\begin{equation}
					|\phi(t,Y)-\sum_{n=0}^{N-1}a_n(Y)t^n|\leq MC^N|t|^N{\Gamma(\frac{N}{k}+1)}\quad 
					(t,Y)\in S_{0}(I)\times \Omega \label{1.1}
			\end{equation}
				holds. The totality of functions on 
				$S_{0}(I)\times \Omega$ that are $k$-Borel summable  with respect to $t$ is denoted by 
				$\mathscr O_{\{1/k\}}(S_{0}(I)\times \Omega)$.   
		\end{defn} 
			We note that  $\phi(0,Y)=0$ means $a_0(y)=0$. 
		 \begin{defn} Let 	
				$\widetilde{\phi}(t,Y)=\sum_{n=0}^{\infty}a_n(Y)t^n \in {\mathscr O}(\Omega)[[t]]$. 
				$\widetilde{\phi}(t,Y)$ is said be 
				$k$-Borel summable in a direction $\theta$ in $\Omega$, if there exists  
				an interval $I=(\theta-\delta, \theta+\delta)$ with $\delta>\pi/2k$ and
				$\phi(t,Y)\in {\mathscr O}_{\{1/k\}}(S_{0}(I)\times \Omega)$ such that 
			\begin{equation}
				|\phi(t,Y)-\sum_{n=0}^{N-1}a_n(Y)t^n|\leq MC^N|t|^N{\Gamma(\frac{N}{k}+1)}\quad 
				(t,Y)\in S_{0}(I)\times \Omega. \label{1.1}
			\end{equation}		
		\end{defn}
		\begin{rem}
			  If $\widetilde{\phi}(t,Y)$ is $k$-Borel summable, then $\phi(t,Y)$ is uniquely determined. Hence
			  $k$-Borel summable $\widetilde{\phi}(t,Y)$ and $\phi(t,Y)$ are often identified 
		\end{rem}
	        Let us proceed to define $k$-Laplace and $k$-Borel transform.
		\begin{defn} Let $k>0$, $\widehat{I}$ be an interval and $\Omega=\{Y\in {\mathbb C}^n;	|Y|<R \}$.  
	   		${\rm Exp}(k,S(\widehat{I})\times \Omega)$ is the totality of 
			    $\psi(\xi,Y)\in \mathscr{O}(S(\widehat{I})\times \Omega)$ such that there exist positive constants
			     $c$ and $M$ with 
		    \begin{equation}
		    	\begin{aligned}
		    	 |\psi(\xi,Y)|\leq M e^{c|\xi|^{k}} \quad  \,\{\xi\in S(\widehat{I});|\xi|\geq 1\}. 	
		    	\end{aligned}
		    \end{equation}
	    \end{defn} 	    
	    	For $\psi(\xi,Y)\in {\rm Exp}(k,S(\widehat{I})\times \Omega)$  with
	    	\begin{equation}
		 	 |\psi(\xi,Y)|\leq M|\xi|^{s-k} \;s>0
		 	 \quad  \, \{\xi\in S(\widehat{I});|\xi|\leq 1\} \label{1.3}
		\end{equation}
			$k$-Laplace transform $\mathscr{L}_{k}\psi$ is defined by  
		  \begin{equation}
		  	\begin{aligned}
		  		(\mathscr{L}_{k,\theta}\psi)(t,Y):=\int_{0}^{e^{i\theta}\infty}\exp({-(\frac{\xi}{t})^{k}})
		  		\psi(\xi,Y)d\xi^{k} \quad  \theta\in \widehat{I}.
		  	\end{aligned}
		  \end{equation}
	    		  Next let us define $k$-Borel transform.  Let $I=(a,b)$ with $b-a>\pi/k$,
   			  $\theta=(a+b)/2$ and $\delta=(b-a)/2>\pi/(2k)$.
   			  Then $I=(\theta-\delta,\theta+\delta)$.
  	             Let  $\psi(t,Y)\in \mathscr{O}(S_0(I)\times \Omega)$ and $|\psi(t,Y)|\leq C|t|^c (c>0)$.  
			$k$-Borel transform of $\psi(x,Y)$  is defined by
		\begin{equation}
			\begin{aligned}
				(\mathscr{B}_{k,\theta}\psi)(\xi,Y):= \frac{1}{2\pi i} \int_{\mathcal C}\exp(\frac{\xi}{t})^{k} 
				\psi(t,Y)dt^{-k},
			\end{aligned}%
		\end{equation}
					where ${\mathcal C}$ is a contour in $S_0(I)$ that starts from $0e^{i(\theta+\delta')}$ to 
					$r_0e^{i(\theta+\delta')}$ on a segment and next on an arc $|t|=r_0$ to $r_0e^{i(\theta-\delta')}$ 
					and finally on a segment ends at 
					$0e^{i(\theta-\delta')}$\; ($\delta>\delta'>\pi/2k$).  		
					Let $\widehat{a}(k)=a+\pi/2k$,  
						$\widehat{b}(k)=b-\pi/2k$ and 
						$\widehat{I}({k})=(\widehat{a}(k),\widehat{b}(k))$. Then the followings hold.	
				\begin{enumerate}
				   \item[{\rm (1)}]{\it $(\mathscr{B}_{k,\theta}\psi)(\xi,Y)$ is holomorphic in an infinite sector
						 	$S(\widehat{I}(k))$ with respect to $\xi$.}
				   \item[{\rm (2)}]  {\it $(\mathscr{B}_{k,\theta}\psi)(\xi,Y)\in
					   {\rm Exp}(k,S(\widehat{I}({k})_{\epsilon}), \Omega)$ for any small $\epsilon>0$ 
				   and $\psi(t,Y)=(\mathscr{L}_{k,\theta}\frak{B}_{k,\theta}\psi)(t,Y)$.}
	  			\end{enumerate}	 
	  			   In particular, the following holds  for 
	  			   $\phi(t,Y)\in {\mathscr O}_{\{1/k\}}(S_{0}(I)\times \Omega)$.                
	   	\begin{prop} Let $I=(\theta-\delta,\theta+\delta)$ with $\delta>\pi/2k$ and 
		     	$\widehat{I}=(\theta-\epsilon_*, \theta+\epsilon_*)$, $\epsilon_*:=\delta-\pi/2k>0$. Let
			   $\phi(t,Y)\in {\mathscr O}_{\{1/k\}}(S_{0}(I)\times \Omega)$ with 
		        ${\phi}(0,Y)=0$.  Then $({\mathscr B}_{k,\theta}{\phi})(\xi,Y)$ 
			   is holomorphically extensible to an infinite sector $S(\widehat{I})$ and
			   $\xi^{k-1}({\mathscr B}_{k,\theta}{\phi})(\xi,Y)$ is holomorohic in $\{|\xi|<r\}$.
			   Further for any small $\epsilon>0$ there exist $C,c>0$ such that
		   \begin{equation}
		   	\begin{aligned}
			   	 |({\mathscr B}_{k,\theta}{\phi})(\xi,Y)|\leq C\frac{|\xi|^{1-k}}{\Gamma(1/k)}e^{c|\xi|^k}
			   	 \quad Y\in \Omega
		   	\end{aligned}\label{1.11}
		   \end{equation}
			   holds in  $\big(S(\widehat{I}_\epsilon)\cup \{0<|\xi|<r\}\big)\times \Omega$.  $\phi(t,Y)$ is 
			   represented by
		   \begin{equation}
		   	\begin{aligned}
			   	  \phi(t,Y)=\int_{0}^{\infty e^{i\theta}} \exp(-(\frac{\xi}{t})^k)
			   	  ({\mathscr B}_{k,\theta}{\phi})(\xi,Y)d\xi^k\quad \theta \in \widehat{I}.
		   	\end{aligned}
		   \end{equation}
			   \label{prop1.5}
	   \end{prop}
 	 	 	 We denote $({\mathscr B}_{k,\theta}{\phi})(\xi,Y)$ by $\widehat{\phi}(\xi,Y)$ often in short. 
		   	We have the following from Cauchy's inequality. 
     	\begin{lem} Let $\widehat{J}=(\widehat{a},\widehat{b})$ and 
					$\psi(\xi,Y)\in {\rm Exp}(k,S(\widehat{J})\times \Omega)$ with
				$	 |\psi(\xi,Y)|\leq M|\xi|^{s-k} e^{c|\xi|^{k}}$. Let 
				$\psi(\xi,Y)=\sum_{q\in \mathbb{N}^n}\psi_q(\xi)Y^q$. Then 
			\begin{equation}
				\begin{aligned}
				     |\psi_q(\xi,Y)|\leq \frac{M |\xi|^{s-k}e^{c|\xi|^{k}}}{R^{|q|}}.
				\end{aligned}
			\end{equation}\label{lem1.6}
	\end{lem}
				$k$-convolution $(\psi_1\underset{k}{*}\psi_2)(\xi,Y)$ is defined by
			\begin{equation}
				\begin{aligned}
				(\psi_1\underset{k}{*}\psi_2)(\xi,Y):=\int_0^{\xi}\psi_1\big((\xi^k-\eta^k)^{1/k},Y\big)
				\psi_2(\eta,Y)d\eta^k.
				\end{aligned}
			\end{equation}
      		\begin{lem} Let $0<k\leq \kappa$ and $\widehat{J}=(\widehat{a},\widehat{b})$. Let  
				$\phi_i(\xi,x)\in {\mathscr{O}}(S_{0}(\widehat{J})\times \Omega)\; (i=1,2)$ such that
			\begin{equation*}
				\begin{aligned}
				 |\phi_i(\xi,x)|\leq \frac{C_i|\xi|^{s_i-k}e^{c|\xi|^{\kappa}}}{\Gamma(s_i/k)}, \; s_i>0.
				 \end{aligned}
			\end{equation*}
				Then $(\phi_1\underset{k}{*}\phi_2)(\xi,x) \in  {\mathscr{O}}(S_{0}(\widehat{J})\times \Omega)$ 
					and  
			\begin{equation}	
				\begin{aligned}
				   |(\phi_1\underset{k}{*}\phi_2)(\xi,x)|\leq 
				    \frac{C_1C_2|\xi|^{s_1+s_2-k}e^{c|\xi|^{\kappa}}}{\Gamma((s_1+s_2)/k)}. 
				\end{aligned}
			\end{equation}	\label{conv}
		\end{lem}
		\begin{proof}
			Let $\arg \xi =\theta$. Then we have  
			\begin{equation*}
				\begin{aligned}
					 (\phi_1\underset{k}{*}\phi_2)(\xi,x)&=\int_{0}^{|\xi|{e^{i\theta}}}
							 	\phi_1((\xi^k-\eta^k)^{1/k},x)\phi_2(\eta,x)d\eta^k \\
						&=\int_{0}^{|\xi|}
							 	\phi_1((|\xi|^k-r^k)^{1/k}e^{i\theta},x)\phi_2(re^{i\theta},x)e^{ik\theta}  dr^k	 	
				\end{aligned}
			\end{equation*}
			and it follows from $(|\xi|^k-r^k)^{\kappa/k}+r^{\kappa}\leq  |\xi|^{\kappa}$ for
				$0\leq r \leq |\xi|$ that
			\begin{equation*}
				\begin{aligned}
					 |(\phi_1\underset{k}{*}\phi_2)(\xi,x)|
					& \leq \int_{0}^{|\xi|}
							 	|\phi_1((|\xi|^k-r^k)^{1/k}e^{i\theta},x)||\phi_2(re^{i\theta},x)| dr^k \\
					& \leq \frac{C_1C_2e^{c|\xi|^{\kappa}}}{\Gamma(s_1/k)\Gamma(s_2/k)} \int_{0}^{|\xi|}
							 	(|\xi|^k-r^k)^{s_1/k-1}r^{s_2-k} dr^k	\\
					& \leq \frac{C_1C_2|\xi|^{s_1+s_2-k}e^{c|\xi|^{\kappa}|}}{\Gamma((s_1+s_2)/k)}.			 		 		 				\end{aligned}
			\end{equation*}
		\end{proof}
			We remark that Lemma \ref{conv} also holds for an infinite sector $S(\widehat{J})$.
		Let $\varphi_i(t,x)=({\mathscr L}_\theta \phi_i)(t,x)$ (i=1,2). Then there is 
		a basic relation between product of functions and convolution,
		\begin{equation}
			\begin{aligned}
			\varphi_1(t,x)\varphi_2(t,x)=\big({\mathscr L}_\theta (\phi_1*\phi_2)\big)(t,x).
			\end{aligned}\label{prod}
		\end{equation} 
				We sum up shortly relations between formal power series  $\mathscr{O}(\Omega)[[t]]$
					and $\mathscr{O}_{\{1/k\}}(S_0(I)\times\Omega)$.
		  \begin{defn} Let $k>0$. ${\mathscr O}(\Omega)[[t]]_{\frac{1}{k}}$ is the set of totality of 
  	 	   	$\widetilde{\psi}(t,Y)=\sum_{n=0}^{\infty}c_n(Y)t^n \in {\mathscr O}(\Omega)[[t]]$ 
  		    	such that there exist constants
		      	$M$ and $C$ with $|c_n(Y)|\leq MC^n\Gamma(n/k+1)$ for $Y\in \Omega$.
		    \end{defn}
 		  	Let $	\widetilde{\psi}(t,Y)=\sum_{n=1}^{\infty}c_n(Y)t^n \in {\mathscr O}(\Omega)[[t]]  $ with
			$c_0(Y)=0$.
   		Then formal $k$-Borel transform $\widetilde{{\mathscr B}}_{k}\widetilde{\psi}$ is defined by
	     \begin{equation}
        	\begin{aligned}
	        	  (\widetilde{{\mathscr B}}_{k}\widetilde{\psi})(\xi,Y):=\sum_{n=1}^{\infty}
 	       	  \frac{c_n(Y)}{\Gamma(\frac{n}{k})}\xi^{n-k}.
        	\end{aligned}
  	    \end{equation}
	       If $\widetilde{\psi}(t,Y) \in {\mathscr O}(\Omega)[[t]]_{\frac{1}{k}}$,  then there is a constant  $r>0$ such that
	        $\xi^{k-1}(\widetilde{{\mathscr B}}_{k}\widetilde{\psi})(\xi,Y)$ converges in $\{|\xi|<r\}$. 
	     If $\phi(t,x)\in \mathscr{O}_{\{1/k\}}(S_0(I)\times\Omega)$ with asymptotic expansion 
	      $\widetilde{\phi}(t,Y)=\sum_{n=1}^{\infty}a_n(Y)t^n \in {\mathscr O}(\Omega)[[t]] $, then
	    $({\mathscr B}_{k,\theta}{\phi})(\xi,Y)=(\widetilde{{\mathscr B}}_{k}\widetilde{\phi})(\xi,Y)$.
		\begin{rem}	
			      Let $I=(a,b)$ be an interval with $b-a>\pi/k$ and $\widehat{I}=(a+\pi/2k,b-\pi/2k)$.
			      Let ${\psi}^*(\xi,Y)\in {\rm Exp}(k,S(\widehat{I})\times \Omega)$ such that
			       $\xi^{k-1}\psi^*(\xi,Y)$ is holomorphic at $\xi=0$.  Let 
			  \begin{equation}
			 	  \psi(t,Y)=\int_0^{\infty e^{i\theta}}
			 	  e^{-(\frac{\xi}{t})^k}{\psi}^*(\xi,Y)  d\xi^{k},  \;\theta \in \widehat{I}.   
			 \end{equation}
			  Then $\psi(t,Y)\in {\mathscr O}_{\{1/k\}}(S_{0}(I_{\epsilon})\times \Omega)$ for any $\epsilon>0$, 
			  $\psi(0,Y)=0$ 
				and $\widehat{\psi}(\xi,Y)={\psi}^*(\xi,Y)$.  In order to construct  
			  $\psi(t,Y)\in {\mathscr O}_{\{1/k\}}(S_{0}(I_{\epsilon})\times \Omega)$ in the present paper, 
			  we construct ${\psi}^*(\xi,Y)$ with the above property.     
	   	\end{rem}
	\section{\normalsize  Cauchy Problem in Borel summable function  spaces} 
			First we assume  $f_{i}(t,x,U,p)	\in \mathscr{O}_{\{1/k\}}(S_0(I)\times \Omega)$, that is, 
                    $\{f_{i}(t,x,U,p)\}_{i=1}^{m}$ are $k$-Borel summable. 
		  	  The aim of sections 2 and 3 is to show a solution  $U(t,x)\in \mathscr{O}_{\{1/k\}}(S_0(I')\times \omega)$,
		\;$ I'\subset I, \omega\subset \Omega$.  In other word  a formal solution $\widetilde{U}(t,x)$ (see \eqref{fs}) 
		is $k$-Borel summable.  
			   	The methods of these sections (Borel summability case) are applicable to study mutisummability 
			   	case in the present paper.
 			\subsection{\normalsize Preliminary}
	 	 		 Let us introduce some notations. 
	  	   Let $\Delta=\{(i,j)\in {\mathbb N}^2; 1\leq i \leq m, 1\leq j \leq n\}$, 
	  	   ${\mathbb N}=\{0,1,2,\cdots\}$, $x=(x_1,\dots,x_n)$,	$U=(u_1,u_2,\dots,u_m)$  
				 and $P=\big(p_{i.j}; (i,j)\in \Delta \big)$. Norms are defined by $|x|=\max_{1\leq i \leq n} |x_i|$,
				 $|U|=\max_{1\leq i \leq m} |u_i|$ and 
				 $|P|=\max_{(i,j)\in \Delta} |p_{i,j}|$ respectively. 
				 For $\alpha=(\alpha_1,\alpha_2, \dots,\alpha_m)\in {\mathbb N}^m$, $|\alpha|=\sum_{i=1}^m \alpha_i$ 
				 and $U^{\alpha}=u_1^{\alpha_1}\dots u_m^{\alpha_m}$. Let  
				 $A=\big(A_{i,j}; (i,j)\in \Delta\big)\in \mathbb{N}^{mn}$,  
				 $|A|=\sum_{(i,j)\in \Delta}A_{i,j}$ and $ P^{A}=\prod_{(i,j)\in \Delta} p_{i,j}^{A_{i,j}}$.  
		     Let $I=(-\delta,\delta)$ with $\delta>\pi/2k$, $\Omega_0=\{x \in {\mathbb C}^{n}; |x|<R_0\}$,  
		      $\Omega'=\{(U,P); |U|, |P|<R_1 \}$ and $\Omega=\Omega_0\times \Omega'$. 
		      By reducing the initial values 
	       $\{u_i(0,x)\}_{i=1}^{m}$ to 0, we study the following Cauchy problem (initial value problem)
	        of a system of nonlinear partial differential equations in a domain $S_0(I)\times \Omega_0$, 
	   \begin{equation} {\rm (CP)}\quad \left \{
			\begin{aligned} \
				 & \partial_t u_i=f_i\big(t,x, U, \nabla_xU \big) \quad u_i(0,x)= 0 \quad 1 \leq i \leq m, \\
				 & U=(u_1, \cdots,u_m)\quad  \nabla_x U =\{\partial_{x_j}u_i\}_{1\leq i \leq m, 1\leq j \leq n}.
			\end{aligned} \right .\label{CP} 
		\end{equation}
		  As mentioned in Introduction the unique existence of solution follows from \cite{Nagumo}. As for 
		  asymptotic expansion we get 
		\begin{thm}
			Suppose $f_i\big(t,x,U,P \big)\in \mathscr{O}_{\{1/k\}}(S_0(I) \times \Omega)$ $(1\leq i \leq m)$.
			Let $U(t,x)=(u_1(t,x),\dots,u_m(t,x))$ be a solution of {\rm (CP)}. Then 
			there exists a neighborhood $\omega$ of $x=0$ such that  
			$u_i(t,x)\in \mathscr{O}_{\{1/k\}}(S_0(I_\epsilon) \times \omega)$ $(1\leq i \leq m)$. 
			for any small $\epsilon>0$. \label{th-1}
		\end{thm}  
			Hence the unique solution $U(t,x)=(u_1(t,x),\dots,u_m(t,x))$ of (CP)   
			is $k$-Borel summable with asymptotic expansion $\widetilde{u}_i(t,x)$ 
			at $t=0$ in $S_0(I_\epsilon)\subset S_0(I) $.
		  We study \eqref{CP} in Borel summable functions space in detail in this section. 
		  Cauchy problem \eqref{CP} in multi-summable function spaces is studies in sections 4 and 5. 
			\par 
				We proceed to show Theorem \ref{th-1}. 
	  		We further transform \eqref{CP}. Let $u_i(t,x)=0+\partial_t u_i(0,x)t+ w_i(t,x),$ $w_i(t,x)=O(t^2)$. 
			We get a system of partial differential equations with unknown functions $(w_i,\dots,w_m)$. 
			We denote $w_i(t,x)$ by $u_i(t,x)$ and the right hand side of the equations by 
			$f_i\big(t,x, U, \nabla_xU \big)$ again. Then $u_i(0,x)=\partial_t u_i(0,x)= 0 \;(1 \leq i \leq m)$
			($u_i(t,x)=O(t^2)$) and 
		\begin{equation}
			\begin{aligned} \
				 & \partial_t u_i=f_i\big(t,x, U, \nabla_xU \big)\quad  f_i(0,x, 0, 0)=0 \quad 1\leq i \leq m.
			\end{aligned}  
		\end{equation}										
		 		Since  $u_i(t,x)=O(t^2)$, we can set $u_i(t,x)=tv_i(t,x)$ with $v_i(0,x)=0$ and have	
		 	\begin{equation}
		 		\begin{aligned}
		 		   \partial_t (t v_i)=f_i(t,x, tV, t \nabla_xV )\;\; V=(v_1,v_2,\dots,v_m). 
		 		\end{aligned} \label{CP-0}
		 	\end{equation}
		 	There exists a unique formal solution ${\widetilde V}=
		 		({\widetilde v}_1(t,x),\dots ,{\widetilde v}_m(t,x))\in
		 	{\mathscr O}(\Omega_0)[[t]]^m$ of \eqref{CP-0}
		 	\begin{equation}
		 		\begin{aligned}
		 		{\widetilde v}_i(t,x)=\sum_{n=1}^{\infty}v_{i,n}(x)t^n.
		 		\end{aligned} \label{fv}
		 	\end{equation}
		 	Let $g_i(t,x, V, P):=f_i(t,x, tV, tP)$.  By expansion of $f_i(t,x, U, P)$ at $(U,P)=(0,0)$, we have
		  \begin{equation} \left \{
					\begin{aligned} \ &
						f_i(t,x, U, P)= \sum_{(\alpha,A)\in {\mathbb N}^m\times {\mathbb N}^{mn}}f_{i,\alpha,A}(t,x) 
						U^{\alpha} P^{A}, \\ 
						& f_{i,\alpha,A}(t,x)\in \mathscr{O}_{\{1/k\}}(S_0(I)\times \Omega_0 ), \\
						&  g_i(t,x, V, P)=\sum_{(\alpha,A)\in {\mathbb N}^m\times {\mathbb N}^{mn}
						}g_{i,\alpha,A}(t,x)V^{\alpha}P^A, \\
						&  g_{i,\alpha,A}(t,x)=t^{|\alpha|+|A|}f_{i,\alpha,A}(t,x).
					\end{aligned}\right.
	  	\end{equation}			
		 	 It holds that $g_{i,0,0}(t,x,V,P)=f_i(t,x,0,0)$ and $g_{i,\alpha,A}(0,x)=0$.
		 \begin{lem}\; 
		 				$g_i(t,x, V, P)\in \mathscr{O}_{\{1/k\}}(S_0(I)\times \Omega)$ .
		 				\label{lem2.2}
		 	\end{lem}	
		 	We refer Lemma \ref{lem2.2} to Appendix. Let
		 	\begin{equation}
		 		\begin{aligned}
		 		\widehat{g}_i(\xi,x, V, P):=(\mathscr{B}_{k,0}	g_i)(\xi,x, V, P))=
		 		\sum_{\alpha,A} \widehat{g}_{i,\alpha,A}(\xi,x)V^{\alpha}P^A.
		 		\end{aligned}
		 	\end{equation}
		 	\begin{lem} Let $\widehat{I}=(-\delta+\pi/2k, \delta-\pi/2k)$.
		 	  \begin{enumerate}
		 	   \item[{\rm (1)}] There exists $r>0$ such that 
			 	   $\xi^{k-1}\widehat{g}_{i,\alpha,A}(\xi,x)$ is holomorphic in 
			 	   $\{(\xi,x)\in  (S(\widehat{I})\cup \{|\xi|<r\})\times \Omega_0 \}$
			 	\item[{\rm (2)}]
				 	  For any $\epsilon>0$ there exist constants $G, C,c>0$ 
 			     such that for $(\xi,x)\in  \big(S(\widehat{I}_\epsilon)\cup \{0<|\xi|<r\}\big)\times \Omega_0 $
			 	\begin{equation} \left \{
				  \begin{aligned}
				  	\ & |\widehat{g}_{i,0,0}(\xi,x)|\leq G\frac{|\xi|^{1-k}e^{c|\xi|^{k}}}
							{\Gamma(\frac{1}{k})} \\
					 	 &	|\widehat{g}_{i,\alpha,A}(\xi,x)|
					 		\leq G C^{|\alpha|+|A|}\frac{|\xi|^{|\alpha|+|A|-k}e^{c|\xi|^{k}}}
							{\Gamma(\frac{|\alpha|+|A|}{k})} \quad |\alpha|+|A|\geq 1.
					\end{aligned}\right .\label{eq2.7}
				\end{equation}
		 	 \end{enumerate}
		 				\label{lem2.3}
		 	\end{lem}	
		 	The details of Lemmas \ref{lem2.2} and \ref{lem2.3} are given in Appendix. 
			 	Consequently we study the following equation
			\begin{equation}
				\begin{aligned} \
				 & \partial_t (t v_i)=
				 \sum_{\ell=0}^{\infty}\big(\sum_{\begin{subarray}\ (\alpha,A)\\ |\alpha|+|A|=\ell
				 \end{subarray}}g_{i,\alpha,A}(t,x) 
				  V^{\alpha} (\nabla_x V)^{A}\big) \;\; 1\leq i \leq m.
				\end{aligned}  \label{CP-1}
			\end{equation}
  			  We give a formula for calculations in the following sections.
		\begin{lem} Let $\widehat{J}=(a,b)$,\; $c(\xi)\in \mathscr{O}(S(\widehat{J}))$ with 
			$|c(\xi)|\leq M|\xi|^{s-k}e^{c|\xi|^{k}} (s>0)$ and   
			\begin{equation*}
				\begin{aligned}
				 w(t)=\int_{0}^{\infty} e^{-(\frac{\xi}{t})^{k}}c(\xi)d\xi^{k}\;\; \xi \in S(\widehat{J}).
				\end{aligned}
			\end{equation*}
			Then $w(t)\in \mathscr{O}(S_{0}(J))$ ${J}=(a-\frac{\pi}{2k},b+\frac{\pi}{2k})$ and 
		\begin{equation}			\begin{aligned}
			 (tw(t))'=\int_{0}^{\infty}e^{-(\frac{\xi}{t})^{k}}
			 \big(\xi \frac{d}{d\xi}+k+1\big)c(\xi)d\xi^{k} .
			\end{aligned}
		\end{equation}	
			\label{lem2.4}
		\end{lem}
		\begin{proof} By changing integral path in $S(\widehat{J})$, we have $w(t)\in \mathscr{O}(S_{0}(J))$
			and
			\begin{equation*}
				\begin{aligned}\ &
				 (tw(t))'=\int_{0}^{\infty} e^{-(\frac{\xi}{t})^{k}}\big(k(\frac{\xi}{t})^{k}+1\big)
				 c(\xi)d\xi^{k}= \int_{0}^{\infty}e^{-\frac{\eta}{t^{k}}}
				 \big(\frac{k\eta}{t^{k}}+1\big)c(\eta^{1/k})d\eta \\
				  = &  \int_{0}^{\infty}e^{-\frac{\eta}{t^{k}}}\Big(k
				 \big(\eta c(\eta^{1/k})\big)'+ c(\eta^{1/k})\Big)d\eta 
				 =\int_{0}^{\infty}e^{-(\frac{\xi}{t})^{k}}
				 \big(\xi \frac{d}{d\xi}+k+1\big) c(\xi)d\xi^{k}. 
				\end{aligned}
			\end{equation*}
		\end{proof}
	\subsection{\normalsize Convolution equation}  
   		  		Let us proceed to study \eqref{CP-1} with unknowns $V=(v_1,v_2,\dots.v_m)$ in function space 
	 		 		$\mathscr{O}_{\{1/k\}}(S_0(I)\times{\Omega_0})$ $I=(-\delta,\delta)$. 
		      We change \eqref{CP-1} to a system of convolution equations. For convolution product 
		      we use a notation  
   	\begin{equation}
			\begin{aligned} &
			 \prod_{1\leq i\leq N}^*\phi_i= \phi_1	\underset{k}{*}\phi_2\underset{k}{*}\cdots \cdots 
			 \underset{k}{*}\phi_N.
			\end{aligned}
		\end{equation} 
       Let $\widehat{I}=(-\delta+\pi/2k,\delta-\pi/2k)$. We represent $v_r(t,x)$ by Laplace integral
        \begin{equation}
    	\begin{aligned}
	    	v_r(t,x): 
	    	=\int_{0}^{\infty e^{i\theta}}e^{-(\frac{\xi}{t})^{k}}\widehat{v}_{r}(\xi,x)d\xi^{k}\quad
	    	 \xi \in S(\widehat{I}).
    	\end{aligned}\label{vr}
    \end{equation}   
	  We transform \eqref{CP-1} to a system of convolution equations   
	\begin{equation}
		\begin{aligned} \ & (\xi \partial_{\xi}+k+1)\widehat{v}_r(\xi,x)\\
		 = &\sum_{\ell=0}^{\infty}\big(\sum_{|\alpha|+|A|=\ell}\widehat{g}_{r,\alpha,A}(\xi,x)
		 \underset{k}{*} \widehat{V}^{*\alpha}\underset{k}{*} 
		  (\nabla_x \widehat{V})^{*A}\big) 
		  \quad 1\leq r\leq m,
		\end{aligned}\label{Ceq}
	\end{equation}
		 by \eqref{prod} and Lemma \ref{lem2.4},  where
				\begin{equation*}
			{g}_{r,\alpha,A}(t,x)=\int_0^{\infty}e^{-(\frac{\xi}{t})^{k}}
			\widehat{g}_{r,\alpha,A}(\xi,x)d\xi^{k}, \quad 
			\widehat{g}_{r,\alpha,A}(\xi,x)=({\mathscr{B}}_k {g}_{r,\alpha,A})(\xi,x),
			\end{equation*}
		\begin{equation*}
			\begin{aligned} 
			  \widehat{V}^{*\alpha}=\widehat{v}_1^{*\alpha_1}*\cdots *\widehat{v}_m^{*\alpha_m}, \quad
			 \widehat{v}_i^{*\alpha_i}=\overset{\alpha_i}{\overbrace{\widehat{v}_i\underset{k}{*}
			    \cdots \cdots  \underset{k}{*}\widehat{v_i}}},			   
		\end{aligned}
		\end{equation*}    
		\begin{equation*}
			\begin{aligned} &
			 (\nabla_x \widehat{V})^{*A}=\prod_{(i,j)\in \Delta}^*\overset{A_{i,j}}{\overbrace{\partial_{x_j} v_i
			\underset{k}{*}\cdots \cdots \underset{k}{*}\partial_{x_j} v_i}}.
								\end{aligned}
		\end{equation*}  
	    The solution $V(t,x)=(v_1(t,x), \dots,v_m(t,x))$ of \eqref{CP-1}  has asymptotic expansion
		$\widetilde{V}(t,x)=(\widetilde{v}_1(t,x), \dots,\widetilde{v}_m(t,x))$,  $ \widetilde{v}_i(t,x)=\sum_{n=1}^{\infty}
		v_n(x)t^n \in {\mathscr O}(U)[[t]]$. 
		 Our purpose is to solve \eqref{Ceq} and get  $(\widehat{v}_1(\xi,x).\dots,\widehat{v}_m(\xi,x))$.  
		For this purpose  we introduce an auxiliary parameter $\varepsilon$ in order to clarify the process of
		construction of $(\widehat{v}_1(\xi,x).\dots,\widehat{v}_m(\xi,x))$,		
	\begin{equation}
		\begin{aligned} \ & (\xi \partial_{\xi}+k+1)\widehat{v}_r(\xi,x)\\
		 	= &  \sum_{\ell=1}^{\infty}\varepsilon^{\ell}
				 \big(\sum_{|\alpha|+|A|=\ell}\widehat{g}_{r,\alpha,A}(\xi,x)\underset{k}{*}  
			  \widehat{V}^{{*} \alpha}\underset{k}{*}  (\nabla_x \widehat{V})^{* A}\big)
			  +\varepsilon\widehat{g}_{r,0,0}(\xi,x).
		\end{aligned}\label{Conv-e}
	\end{equation}
		  If $\varepsilon=1$, we get \eqref{Ceq}.
   		Let $\widehat{v}_r(\xi,x,\varepsilon)=\sum_{\ell=1}^{\infty} \widehat{v}_{r,\ell}(\xi,x)
  	 	\varepsilon^{\ell}$ and   
 	  	$\widehat{V}(\xi,x,\varepsilon)
	   	=(\widehat{v}_1(\xi,x,\varepsilon),\dots,\widehat{v}_m(\xi,x,\varepsilon))$. 
	   	By substituting ${\widehat{V}=\widehat{V}(\xi,x,\varepsilon)}$ into the right hand side of
	   	\eqref{Conv-e}, we have 
   	\begin{equation}
   		\begin{aligned}
   		 & \sum_{\ell=1}^{\infty}\varepsilon^{\ell}
				\Big[ \Big(\sum_{|\alpha|+|A|=\ell}\widehat{g}_{r,\alpha,A}(\xi,x)\underset{k}{*}  
			  \widehat{V}^{* \alpha}\underset{k}{*}  (\nabla_x \widehat{V})^{* A}\Big)
			  \Big]_{\widehat{V}=\widehat{V}(\xi,x,\varepsilon)}+	\varepsilon\widehat{g}_{r,0,0}(\xi,x) \\
			  &=\sum_{\ell=1}^{\infty}{\mathcal G}_{r,\ell}(\xi,x) \varepsilon^{\ell}, \quad 
			  {\mathcal G}_{r,1}(\xi,x):=\widehat{g}_{r,0,0}(\xi,x)=\widehat{g}_r(\xi,x,0,0)
   		\end{aligned} \label{Conv-e'}
   	\end{equation}
   	and 
   	\begin{equation}
   		\begin{aligned}
   		(\xi\partial_{\xi}+k+1)\widehat{v}_{r,\ell}(\xi,x)={\mathcal G}_{r,\ell}(\xi,x)\quad 
   		 \ell=1,2,\dots 
   		  		\end{aligned}
   	\end{equation}
	   	for $1 \leq r \leq m$. 
	   	 ${\mathcal G}_{r,\ell}(\xi,x)$  is determined by terms 
		 $\{\widehat{v}_{i,j}(\xi,x); 1\leq i\leq m,j<\ell\} $ and is given in detail 
		 in the next section 2.2.1.
   	\subsubsection{\normalsize Calculation of ${\mathcal G}_{r,\ell}(\xi,x)$ }
  	 	We calculate ${\mathcal G}_{r,\ell}(\xi,x)$ more detail. 
	   	We have 
	 \begin{equation*}
	 	\begin{aligned} \ 
		  &	 \widehat{v_i}^{*\alpha_i}=\overset{\alpha_i}{ \overbrace{ \widehat{v}_i\underset{k}{*} \widehat{v}_i
  		 		 \underset{k}{*} \cdots \underset{k}{*} \widehat{v}_i}}\\ = &	 	
			 	 	\big(\sum_{\ell_{i}(1)=1}^{\infty} \widehat{v}_{i,\ell_{i}(1)}(\xi,x)\varepsilon^{\ell_{i}(1)}\big)
		 	 		\underset{k}{*}  	\big(\sum_{\ell_i(2)=1}^{\infty} \widehat{v}_{i,\ell_i(2)}(\xi,x)
		 	 		\varepsilon^{\ell_i(2)}\big)\underset{k}{*} \cdots \underset{k}{*} 
		 			\big(\sum_{\ell_i({\alpha_i})=1}^{\infty} \widehat{v}_{i.\ell_i({\alpha_i})}(\xi,x)
		 			\varepsilon^{\ell_i({\alpha_i})}\big) \\
		 	= & \sum_{\ell'=\alpha_i}^{\infty} \varepsilon^{\ell'}\big(
	 		 		 \sum_{ (\alpha_i,\ell')}v_{i.\ell_i(1)}(\xi,x)\underset{k}{*} v_{i,\ell_{i}(2)}(\xi,x)
	 		 		 \underset{k}{*}  \cdots 
	 		  	\underset{k}{*} v_{i,\ell_{i}(\alpha _i)}(\xi,x)\big),
	 	\end{aligned}
	 \end{equation*} 
		 where $\sum_{ (\alpha_i,\ell')}=\{(\ell_{i}(1),\ell_{i}(2),\dots,\ell_{i}(\alpha_i));\; 
	 		\sum_{s=1}^{\alpha_i}\ell_{i}(s)=\ell'\}$. Hence for $\alpha=(\alpha_1,\dots,\alpha_m)$
 	 \begin{equation}
	 	\begin{aligned}\ 
	 			&	{\widehat{V}}^{*\alpha}=\prod_{ \ 1\leq i \leq m }^*\widehat{v_i}^{*\alpha_{i}}=
			 		\prod_{ \ 1\leq i \leq m }^*
			 		\overset{\alpha_i}{ \overbrace{ \widehat{v}_i\underset{k}{*} \widehat{v}_i
				 	 \underset{k}{*} \cdots \underset{k}{*} \widehat{v}_i}}\\	 	
	 			& = \sum_{\ell'=|\alpha|}^{\infty} \varepsilon^{\ell'}
		 			\Big( \sum_{ (\alpha,\ell')} 
				  \prod^*_{\begin{subarray} \ 1\leq i \leq m 
	 	\end{subarray}}
	 	 	v_{i,\ell_{i}(1)}(\xi,x)\underset{k}{*} v_{i,\ell_{i}(2)}(\xi,x)\underset{k}{*} 
	 	 \cdots 
	 		  \underset{k}{*} v_{i,\ell_{i}(\alpha_{i})}(\xi,x)\Big) \\
	 		  &  \sum_{ ({\alpha},\ell')}=\{(\ell_{i}(1),\ell_{i}(2),\cdots,\ell_{i}(\alpha_{i}))
	 		  \; 1\leq i \leq m; 
	 			  \sum_{i=1}^m
	 		\sum_{s=1}^{\alpha_i}\ell_{i}(s)=\ell'\}.	  
	 	\end{aligned}
	 \end{equation}	
	 	 We also get in the same way for $A=(A_{i,j};(i,j)\in \Delta)\in {\mathbb N}^{mn}$
	 	 \begin{equation*}
		 	\begin{aligned}\ &
		 	 (\partial_{x_j}\widehat{v_i})^{{*} A_{i,j}}=\overset{A_{i,j}} {\overbrace{
			  \partial_{x_j}\widehat{v}_{i}(\xi,x)\underset{k}{*}
		 	 \partial_{x_j}\widehat{v}_i(\xi,x)\underset{k}{*}  \cdots 
	 		  \underset{k}{*} \partial_{x_j}\widehat{v}_{i}(\xi,x)}},  \\
	 		& {\nabla _x \widehat{V}}^{*A}= \prod^*_{\begin{subarray} \ (i,j)\in \Delta 	\end{subarray}}
	 	 (\partial_{x_j}\widehat{v_i})^{{*} A_{i,j}}  
		 	\end{aligned} .
	 \end{equation*}	
	 	 and 
	 	\begin{equation*}
		 	\begin{aligned}\ &
		  \prod^*_{\begin{subarray} \ (i,j)\in \Delta 	\end{subarray}}
	 	 (\partial_{x_j}\widehat{v_i})^{* A_{i,j}}= 
		 	\sum_{\ell''=|A|}^{\infty} \varepsilon^{\ell''}
		 	\Big( \sum_{ (A,\ell'')} \\ &
			  \prod^*_{\begin{subarray} \ (i,j)\in \Delta	\end{subarray}}
			  \partial_{x_j}\widehat{v}_{i,\ell_{i,j}(1)}(\xi,x)\underset{k}{*}
		 	 \partial_{x_j}\widehat{v}_{i,\ell_{i,j}(2)}(\xi,x)\underset{k}{*}  \cdots 
	 		  \underset{k}{*} \partial_{x_j}\widehat{v}_{i,\ell_{i,j}(A_{i,j})}(\xi,x)\Big), \\
	 		  &  \sum_{ (A,\ell'')}=\{(\ell_{i,j}(1),\ell_{i,j}(2),\cdots,\ell_{i,j}(A_{i,j}))\; (i,j)\in \Delta 
	 		 ;  \sum_{ (i,j)\in \Delta 	}
	 		\sum_{s=1}^{A_{i,j}}\ell_{i,j}(s)=\ell''\}.	  
	 	\end{aligned}
	 \end{equation*}	 
	  Thus we get for $|\alpha|+|A|\geq 1$
		\begin{equation}
			\begin{aligned}\ &
				\varepsilon^{|\alpha|+|A|}
				\widehat{g}_{r,\alpha,A}(\xi,x)\underset{k}{*} 
			 \widehat{V}^{* \alpha}
				\underset{k}{*} {\nabla _x \widehat{V}}^{* A}=
				\widehat{g}_{r,\alpha,A}(\xi,x) \sum_{\ell \geq 2|\alpha|+|A|}\
				\varepsilon^{\ell}\bigg( \sum_{ \ell'+\ell''+|\alpha|+|A|=\ell} \\	
				& 	\underset{k}{*}	\Big ( \sum_{ ({\alpha},\ell')}
	 		\prod_{1\leq i\leq m}^*\big(\widehat{v}_{i,\ell_i(1)}(\xi,x)
	 		\underset{k}{*} \widehat{v}_{i.\ell_{i}(2)}(\xi,x)\underset{k}{*}  \cdots 
	 		  \underset{k}{*}\widehat{ v}_{i,\ell_{i}(\alpha_i)}(\xi,x)\big) \Big)\\
	 	& \underset{k}{*} 	\Big (\sum_{ (A,\ell'')} \prod_{\begin{subarray}
	 			 \ (i,j)\in \Delta 	\end{subarray}}^{*} \partial_{x_j}\widehat{v}_{i,\ell_{i,j}(1)}(\xi,x)
	 			\underset{k}{*} \partial_{x_j}\widehat{v}_{i.\ell_{i,j}(2)}(\xi,x)\underset{k}{*}  \cdots 
	 		   \underset{k}{*} \partial_{x_j}\widehat{v}_{i.\ell_{i,j}(A_{i,j})}(\xi,x)\Big)\bigg).   
			\end{aligned}\label{GL}
		\end{equation}
				${\mathcal G}_{r,\ell}(\xi,x)\; (\ell \geq 2)$ is a finite sum of terms 
		\begin{equation}
		\begin{aligned}\ & 
				\widehat{g}_{r,\alpha,A}(\xi,x)\underset{k}{*} \big (
			 	\prod_{1\leq i \leq m}^* \big(\widehat{v}_{i.\ell_i(1)}(\xi,x)\underset{k}{*}
			 	 \widehat{v}_{i.\ell_{i}(2)}(\xi,x)\underset{k}{*}  \cdots 
	 		  \underset{k}{*} \widehat{v}_{i.\ell_{i}(\alpha_i)}(\xi,x)\big) \big)\\
	 	& \underset{k}{*}	\Big ( \prod_{ (i,j)\in\Delta }^{*} \partial_{x_j}\widehat{v}_{i.\ell_{i,j}(1)}(\xi,x)
		 	\underset{k}{*} \partial_{x_j}\widehat{v}_{i.\ell_{i,j}(2)}(\xi,x)
		 	\underset{k}{*} \cdots  \underset{k}{*} \partial_{x_j}\widehat{v}_{i.\ell_{i,j}(A_{i,j})}(\xi,x)\Big) 
		\end{aligned}\label{GL-1}
	\end{equation}
			satisfying 
		\begin{equation}
			\begin{aligned}
			\sum_{i=1}^m\sum_{s=1}^{\alpha_i}\ell_i(s)+\sum_{\begin{subarray}\ (i,j) \in \Delta
			 \end{subarray}}\sum_{s=1}^{A_{i,j}}\ell_{i,j}(s)+|\alpha|+|A|=\ell.
						\end{aligned}\label{GL-2}
		\end{equation}
		 ${\mathcal G}_{r,\ell}(\xi,x)$ is determined by terms 
		 $\{\widehat{v}_{i,j}(\xi,x); 1\leq i\leq m,j<\ell\} $.
    Thus we get equations 
    	\begin{equation}
		\begin{aligned}\
		  & (\xi \partial_{\xi}+k+1)\widehat{v}_{r,\ell}(\xi,x)={\mathcal G}_{r,\ell}(\xi,x) 
		\end{aligned}
		\end{equation}
		and from ${\mathcal G}_{r,1}(\eta,x)=\widehat{g}_r(\xi,x,0,0)$ 
		\begin{equation}\left \{
			\begin{aligned}
					& \widehat{v}_{r,1}(\xi,x)=\xi^{-k-1}\int_{0}^{\xi} \eta^{k}\widehat{g}_r(\eta,x,0,0) d\eta, \\
			    &  \widehat{v}_{r,\ell}(\xi,x)=\xi^{-k-1}\int_{0}^{\xi} \eta^{k}{\mathcal G}_{r,\ell}(\eta,x) d\eta
			    \quad \ell\geq 2.
			\end{aligned}\right .
		\end{equation}
			$\widehat{v}_{r,\ell}(\xi,x)\; (\ell \geq 2)$ are successively determined and 
			we have from Lemma \ref{conv} and Lemma \ref{lem2.3}-(1) 
		\begin{prop}
				$\xi^{k-1}\widehat{v}_{r,\ell}(\xi,x)$ is holomorphic in
				 $ (S(\widehat{I})\cup\{|\xi|<r\})\times \Omega_0$.  \label{prop2.5}		  
		\end{prop}
 	 We give a lemma used for estimates of $\widehat{v}_{r,\ell}(\xi,x)$ in section 3.  
	\begin{lem} Let $g(\xi,x) \in {\rm Exp}(k,S\times U)$ with $|g(\xi,x)|\leq M|\xi|^{p-k}e^{c|\xi|^{k}}$
		\; $(p>0)$	and  
	 		 $$v(\xi,x)=\xi^{-k-1}\int_{0}^{\xi} \eta^{k}g(\eta,x) d\eta.$$
	 		Then $(\xi \partial_{\xi}+k+1)v(\xi,x)=g(\xi,x)$ and 
		  	$ |v(\xi,x)|\leq \frac{M|\xi|^{p-k}}{p+1}e^{c|\xi|^{k}}.$ \label{lem2.6}
	\end{lem}
		\begin{proof} 
		  	From the assumption $ |\eta^{k}g(\eta,x)|\leq M|\eta|^{p}e^{c|\eta|^{k}}$. Then
			\begin{equation*}
				\begin{aligned}\ &                 
				 |\xi^{-k-1}\int_{0}^{\xi} \eta^{k}g(\eta,x)d\eta \leq 
				 M|\xi|^{-k-1}e^{c|\xi|^{k}}\int_{0}^{|\xi|}r^{p}dr\; =\frac{M|\xi|^{p-k} e^{c|\xi|^{k}}}{p+1}.
				\end{aligned}
			\end{equation*}
			We can easily show $(\xi \partial_{\xi}+k+1)v(\xi,x)=g(\xi,x)$.
		\end{proof}
	\section{\normalsize  Estimate and convergence of $\widehat{v}_r(\xi,x,\varepsilon) 
	 	   =\sum_{\ell=1}^{\infty} \widehat{v}_{r,\ell}(\xi,x)\varepsilon^\ell$ }
				We have to estimate $\widehat{v}_{r,\ell}(\xi,x)$ as a function of $(\xi,x)$ 
	 	    to show convergence of $\widehat{v}_r(\xi,x,\varepsilon)
	 	   =\sum_{\ell=1}^{\infty} \widehat{v}_{r,\ell}(\xi,x)\varepsilon^\ell$.
	   	    The main estimate is Proposition \ref{prop-main}. Theorem \ref{th-1} easily follows from it.  
   	\subsection{\normalsize Majorant functions } 
	    We introduce majorant functions to obtain estimates with respect to variable $x$. 
	 		   Let $A(\tau)=\sum_{n=0}^{\infty}a_{n}\tau^{n},
	 		    B(\tau)=\sum_{n=0}^{\infty}b_{n}\tau^{n} \in 
			   {\mathbb C}[[\tau]]$. Then
			   $B(\tau) \ll A(\tau)$ means $|b_{n}|\leq a_{n}$ for all 
			   $n$ and 
			   $A(\tau)\gg 0$ means $a_{n}\geq 0$ 
			   for all $n$. Let 
	  	\begin{equation}
	  		\begin{aligned}
		 		 	\theta(\tau):=\sum_{n=0}^{\infty}\frac{c\tau^n}{(n+1)^3}\quad c>0.
	  		\end{aligned} \label{theta}
			 \end{equation}
			 Constant $c>0$ in $\theta(\tau)$ is fixed later. $\theta(\tau)$ converges in $\{|\tau|\leq 1\}$. 
			 We give properties of $\theta(\tau)$. 
			 	  \begin{lem}
		   {\rm (1)} There exists a constant $C>0$ such that $\theta(\tau)\ll C\theta'(\tau)$. \\
		   {\rm (2)} There exists a constant $c>0$ such that
	  	\begin{equation}																					  		\begin{aligned}																				  		\theta (\tau)\theta(\tau)\ll \theta(\tau),\; \theta (\tau)\theta'(\tau)\ll \theta'(\tau),\;
	  		 \theta'(\tau)\theta'(\tau)\ll \theta'(\tau).
	  		\end{aligned} \label{eq3.2}
	  	\end{equation}
	  \end{lem}
	  \begin{proof}  {\rm (1)} We have
	    	$\theta'(\tau)=c\sum_{n=0}^{\infty}\frac{n+1}{(n+2)^3}\tau^n $. Take $C>0$ such that 
	    	$\frac{1}{(n+1)^3}\leq C\frac{n+1}{(n+2)^3}$ for all $n\geq 0$.\\
    	{\rm (2)} We show the third estimate of \eqref{eq3.2}. We have 
	  	\begin{equation*}
	  	\begin{aligned}
	  				\theta'(\tau)\theta'(\tau)=c^2\sum_{n=0}^{\infty}\big(\sum_{k+\ell=n}\frac{(k+1)(\ell+1)}
	  				{(k+2)^3(\ell+2)^3 }\big)\tau^n.
	  		\end{aligned}
	  	\end{equation*}
	  	There are constants $C_1, C_2>0$ such that
	 	  \begin{equation*}
	  		\begin{aligned}\ &
	  			\sum_{k+\ell=n}\frac{(k+1)(\ell+1)}{(k+2)^3(\ell+2)^3 }
	  			\leq \sum_{k+\ell=n}\frac{1}{(k+2)^2(\ell+2)^2 } \\ &
	  			=\sum_{0\leq k \leq n/2}\frac{1}{(k+2)^2(n-k+2)^2 } 
	  				+\sum_{n/2<k \leq n}\frac{1}{(k+2)^2(n-k+2)^2 }\\
	  			&	\leq \frac{4}{(n+4)^2}\big(\sum_{0\leq k \leq n/2}\frac{1}{(k+2)^2} 
	  				+\sum_{n/2<k \leq n}\frac{1}{(n-k+2)^2 }\big)\leq  \frac{C_1}{(n+4)^2}
	  			\end{aligned}
	  	\end{equation*}
			  	and $\frac{C_1}{(n+4)^2}\leq \frac{C_2(n+1)}{(n+2)^3}$. Take $c>0$ such that $c^2C_2\leq c$. 
			  	Hence $	\theta'(\tau)\theta'(\tau)\ll \theta'(\tau)$. We can show in the same way that the other 
			  	estimates are valid. 
	  	 \end{proof}
	  	  Majorant function $\theta (\tau)$ satisfying $\theta (\tau)\theta(\tau)\ll \theta(\tau)$ was  
    	 used in Wagschal \cite{Wag}.  
    	 In the following $c>0$ is fixed so that \eqref{eq3.2} holds. 	
         Let $R>0$ and 
      \begin{equation}
      	\begin{aligned}
      	 \Theta(\tau):=\theta(\tau/R).
      	\end{aligned}\label{Theta}
      \end{equation}   
      Then $\Theta^{(\ell)}(\tau)=R^{-\ell}\theta^{(\ell)}(\tau/R)$.	
	  \begin{lem} There is a constant $B>0$ such that for $\ell \in {\mathbb N}$
		\begin{equation}
	    		\begin{aligned}&  	     
	    		\Theta(\tau)\ll \frac{B^{\ell}}{\ell!}{\Theta^{(\ell)}(\tau)}. 
			\end{aligned} \label{eq3.4}
    		\end{equation} \label{lem3.2}
	  \end{lem} 
	  \begin{proof} We have  	
	    	\begin{equation*}
	    		\begin{aligned}&  \frac{\Theta^{(\ell)}(\tau)}{\ell !}= 
	    	 \frac{\theta^{(\ell)}(\tau/R)}{R^{\ell}\ell !}
	    	 =\frac{c}{R^{\ell}}\sum_{n=0}^{\infty}\frac{(n+\ell)!}
	    		{\ell ! n!(n+\ell+1)^3}\frac{\tau^{n}}{R^{n}} 
	    		\gg\frac{c}{R^{\ell}}
	    		\sum_{n=0}^{\infty}\frac{1}
	    		{(n+\ell+1)^3}\frac{\tau^{n}}{R^{n}}	    		
    		\end{aligned}
    		\end{equation*}
    		and
	    		$\frac{1}{({n+\ell+1})^3}\geq 
	    		\frac{1}{(\ell+1)^3}\frac{1}{(n+1)^3}$. 
     	Therefore 
		    \begin{equation*}
		    		\begin{aligned}&  \frac{\Theta^{(\ell)}(\tau)}{\ell !} 
		 		   	\gg \frac{1}{R^{\ell}(\ell+1)^3}
		    		\sum_{n=0}^{\infty}\frac{c}
		    		{(n+1)^3}\frac{\tau^{n}}{R^{n}}.
		    		=\frac{1}{R^{\ell}(\ell+1)^3}\Theta(\tau).
		 				\end{aligned}
	    	\end{equation*} 
			    By taking $B$ such that $R^{\ell}(\ell+1)^3\leq B^{\ell}$, we get \eqref{eq3.4}. 
    		\end{proof}
    	\begin{lem} The following estimates hold.
	   	\begin{equation}
	  		\begin{aligned}
	  		\sum_{\sum_{s=1}^{n}\ell_s=\ell}\frac{\ell!}{\ell_1!\ell_2!\cdots \ell_n!}
	  		\Theta^{(\ell_1)}(\tau)\Theta^{(\ell_2)}(\tau)\cdots\Theta^{(\ell_n)}(\tau)\ll\Theta^{(\ell)}(\tau),
	  		\end{aligned}\label{eq3.5}
	  	\end{equation}
	  	\begin{equation}
	  		\begin{aligned}\ &
	  		\sum_{\sum_{s=1}^{n}\ell_s=\ell}\frac{\ell!}{\ell_1!\ell_2!\cdots \ell_n!}
	  		\Theta^{(\ell_1+1)}(\tau)\Theta^{(\ell_2+1)}(t)\cdots\Theta^{(\ell_n+1)}(\tau)\\ 
	  		 & \ll R^{-n+1} \Theta^{(\ell+1)}(t),
	  		\end{aligned}\label{eq3.6}
	  	\end{equation}
	  	\begin{equation}
	  		\begin{aligned}\ &
	  	 	\sum_{\ell_1+\ell_2=\ell}\frac{\ell!}{\ell_1!\ell_2!}
	  		\Theta^{(\ell_1)}({\tau})\Theta^{(\ell_2+1)}({\tau}) \ll \Theta^{(\ell+1)}(\tau).
	  		\end{aligned}\label{eq3.7}
	  	\end{equation} \label{lem3.3}
	  \end{lem} 
	  \begin{proof} From $\Theta(\tau)^n\ll \Theta(\tau)$ we have 
	  $(\frac{d}{d\tau})^\ell \Theta(\tau)^n\ll (\frac{d}{d\tau})^\ell\Theta(\tau)$ and \eqref{eq3.5}. From
	   $\theta'(\tau/R)^n\ll \theta'(\tau/R)$ we have $R^{-n}\theta'(\tau/R)^n\ll R^{-n} \theta'(\tau/R)$ and
	   $\Theta'(\tau)^n\ll R^{-n+1} \Theta'(\tau)$. 
	  	Hence
	  	\begin{equation*}
	  		\begin{aligned}\ &
	  	  	\sum_{\sum_{s=1}^{n}\ell_s=\ell}\frac{\ell!}{\ell_1!\ell_2!\cdots \ell_n!}
	  		\Theta^{(\ell_1+1)}({\tau})\Theta^{(\ell_2+1)}({\tau})\cdots
	  		\Theta^{(\ell_n+1)}({\tau}) \ll R^{-n+1}\Theta^{(\ell+1)}(\tau).
	  		\end{aligned}
	  	\end{equation*}
	  	From $\theta(\tau/R)\theta'(\tau/R)\ll \theta'(\tau/R)$, we have
	  		\begin{equation*}
	  		\begin{aligned}\ &
	  	  	\sum_{\ell_1+\ell_2=\ell}\frac{\ell!}{\ell_1!\ell_2!}
	  		\Theta^{(\ell_1)}({\tau})\Theta^{(\ell_2+1)}({\tau}) \ll \Theta^{(\ell+1)}(\tau).
	  		\end{aligned}
	  	\end{equation*}
	  \end{proof}
	  	   We have from Lemma  \ref{lem3.3}
	  \begin{cor}	   
	   	\begin{equation}\left \{
	  		\begin{aligned}\ &
	  	\frac{\Theta^{(\ell_1)}(t)\Theta^{(\ell_2)}(t)\cdots\Theta^{(\ell_n)}(t)}{\ell_1!\ell_2!\cdots \ell_n!}
	  	\ll\frac{\Theta^{(\ell)}(t)}{\ell!},
	  	\\ & \frac{
	  		\Theta^{(\ell_1+1)}(t)\Theta^{(\ell_2+1)}(t)\cdots	\Theta^{(\ell_n+1)}(t)}{\ell_1!\ell_2!
	  		\cdots \ell_n!}	\ll \frac{1}{R^{n-1}}\frac{ \Theta^{(\ell+1)}(t)}{\ell!}.
	   		\end{aligned} \right . \label{3.8}
	   	\end{equation}\label{Cor-1}
	\end{cor}
	\subsection{\normalsize Estimate} 
	   		Let $x=(x_1,x_2, \cdots,x_n)$. 
	    For $A(x)=\sum_{\alpha\in {\mathbb N}^n } a_{\alpha}x^{\alpha}$, 
	    $B(x)=\sum_{\alpha\in {\mathbb N}^n } b_{\alpha}x^{\alpha} \in {\mathbb C}
	    [[x]]$, $ B(x)\ll A(x)$ means $|b_{\alpha}|\leq a_{\alpha}$ for 
	    all $\alpha\in {\mathbb N}^n$.     
			Let $X=x_1+x_2+\cdots + x_n$, $R>0$ be a constant such that  
			${\{x\in \mathbb{C}^n; |x|\leq R\}}$ is contained in $\Omega_0$ and $\Theta(X)=\theta(X/R)$.
			$\Theta(X)$ is a convergent power series of $x_1,x_2, \cdots,x_n$ with positive coefficients.
      We have from Lemma \ref{lem2.3} 
 		 \begin{lem} There exist positive constants $M, C_0>0$ such that 	
	   	  \begin{equation}
     			\begin{aligned}\ 
				    	\widehat{g}_{r,\alpha,A}(\xi,x)\ll& MC_0^{|\alpha|+|A|}
				  	  \frac{|\xi|^{|\alpha|+|A|-k}e^{c|\xi|^{k}}}{\Gamma((|\alpha|+|A|)/k)}
				 	   \frac{\Theta^{(|\alpha|+|A|-1)}(X)}{(|\alpha|+|A|-1)!} \\
				 	   & \xi \in S(\widehat{I}_\epsilon)\cup \{0<|\xi|<r\}
			   \end{aligned}\label{g^}
 	  	 \end{equation} 
 	  	 for $|\alpha|+|A|\geq 1$ 
	   \end{lem}
	   \begin{proof} 
	   		It follows from \eqref{eq2.7} that there exists $M>0$ such that 
	   			\begin{equation*}
	   			\begin{aligned} 
			   			& \widehat{g}_{r,\alpha,A}(\xi,x)\ll MC^{|\alpha|+|A|}
			   			\frac{|\xi|^{|\alpha|+|A|-k}e^{c|\xi|^{k}} }{\Gamma(\frac{|\alpha|+|A|}{k})}\Theta(X).
	   			\end{aligned}
	   		\end{equation*}
	   		From 
	   		$\Theta(X)\ll B^{\ell}\frac{\Theta^{(\ell)}(X)}{\ell !}$ (Lemma \ref{lem3.2}) we take
	   		$C_0= BC$ and have \eqref{g^}.
	   \end{proof} 
		   We also have from \eqref{eq2.7}
		   $\widehat{g}_{r,0,0}(\xi,x)= \widehat{g}_{r}(\xi,x,0,0)
		   	\ll M'\frac{|\xi|^{1-k}e^{c|\xi|^{k}}}	{\Gamma(\frac{1}{k})}\Theta^{(1)}(X).$
  			 $\widehat{v}_{r,1}(\xi,x)$ is determined by
 		 \begin{equation*}
			  \widehat{v}_{r,1}(\xi,x)=\xi^{-k-1}\int_{0}^{\xi} \eta^{k}  \widehat{g}_{r}(\eta,x,0,0) d\eta. 
		\end{equation*}
 		There exist $\{M_{r,1}\}_{r=1}^m$  from Lemma \ref{lem2.6} such that 
 	 \begin{equation}
		\begin{aligned}
		 &  \widehat{v}_{r,1}(\xi,x) \ll 
		 \frac{M_{r,1}|\xi|^{1-k}e^{c|\xi|^{k}}}{\Gamma(1/k)}{\Theta^{(1)}(X)}.
		\end{aligned} \label{Mr1}
	\end{equation}
	   It follows from	Proposition \ref{prop2.5}	that 
		$\{\xi^{k-1}\widehat{v}_{r,\ell}(\xi,x)\}_{\ell=1}^{\infty}$ is holomorphic in
		$ (S(\widehat{I})\cup\{|\xi|<r\})\times \Omega_0$. Further we have the following estimates for them,  
	\begin{prop} Let $\epsilon>0$ be a small constant and $\xi \in (S(\widehat{I}_{\epsilon})\cup\{0<|\xi|<r\})$.
			 \\
			{\rm (1)} There exist constants $\{M_{r,\ell}\}_{r=1}^m$ $(\ell=1,2 \dots)$ and $c$ depending on 
			$\epsilon$ such that
			\begin{equation}
				\begin{aligned}
					 \widehat{v}_{r,\ell}(\xi,x)\ll M_{r,\ell}
					 \frac{|\xi|^{\ell-k}e^{c|\xi|^{k}}}{\Gamma(\ell/k)}
					\frac{\Theta^{(\ell)}(X)}{\ell !}
				\end{aligned} \label{M-est}
			\end{equation}
		 	and $\sum_{\ell=1}^{\infty}M_{r,\ell}\tau^{\ell}$ converges in a neighborhood of $\tau=0$. \\
			{\rm (2)}	Let $\varepsilon=1$ and 
			$\widehat{v}_r(\xi,x)=\sum_{\ell=1}^{\infty}\widehat{v}_{r,\ell}(\xi,x)$. 	Then
					$\{\widehat{v}_{r}(\xi,x)\}_{r=1}^m$ converge in a neighborhood $\omega_0$ of $x=0$ and 
					there exist constants $M$ and $c'$ such that 
			\begin{equation}
				\begin{aligned}
				|\widehat{v}_{r}(\xi,x)|\leq M
				{|\xi|^{1-k_1}e^{c'|\xi|^{k}}}.
				\end{aligned} \label{main-est}
			\end{equation} 	\label{prop-main} 
		\end{prop}
		      $\widehat{V}(\xi,x)=(\widehat{v}_{1}(\xi,x),\dots,\widehat{v}_{m}(\xi,x))$ is a
		      solution of convolution equation \eqref{Ceq}.
			Theorem \ref{th-1} follows from Proposition \ref{prop-main}. 
			Let  
			\begin{equation}
	 		   	\begin{aligned}
			    	v_r(t,x): =	{\mathscr L}_{k,0}\widehat{v}_{i}
			    	=\int_{0}^{\infty e^{i\theta}}e^{-(\frac{\xi}{t})^{k}}\widehat{v}_{r}(\xi,x)d\xi^{k}\quad
			    	 \xi \in S(\widehat{I}_{\epsilon}).
	 		   	\end{aligned}	\label{vr'}
		    \end{equation}  
			(see \eqref{vr}) and $V(t,x)=({v}_1(t,x),\cdots,{v}_m(t,x))$. Then $V(t,x)$ satisfies \eqref{CP-1}. 
			This means solution $U(t,x)=tV(t,x)$ of \eqref{CP} is Borel summable. 
			\par  
		\subsection{\normalsize Proof of Proposition \ref{prop-main} \;\; constants $\{M_{r,\ell}\}$} 
		   Proof of Proposition \ref{prop-main} consists of 2 parts. 
		   We show how to determine constants $\{M_{r,\ell}\; \ell\geq 1\}$ in this subsection.
			 The convergence of $\sum_{\ell=1}^{\infty} M_{r,\ell}s^{\ell}$ is shown in the next subsection. \\  
	    \; $\{M_{r,1}\}_{r=1}^m$ are determined by \eqref{Mr1}. 
    We assume $\{M_{r,\ell}\}_{r=1}^m$ with \eqref{M-est} are determined for $1\leq \ell \leq L-1$. 
    Let us show how to determine $\{M_{r,L}\}_{r=1}^m$.  Let
    \begin{equation*}
    	\begin{aligned}
    	 I:=&\prod_{1\leq i\leq m}^{*} \widehat{v}_{i,\ell_i(1)}(\xi,x)
    	 *\widehat{v}_{i,\ell_{i}(2)}(\xi,x)* \ldots 
	 		  *\widehat{v}_{i,\ell_{i}(\alpha_i)}(\xi,x),\\
	 		 II:=& \prod_{\begin{subarray} \ (i,j)\in \Delta \end{subarray}}^*
	 	    \partial_{x_j}\widehat{v}_{i,{\ell_{i,j}(1)}}(\xi,x)*\partial_{x_j}
	 	    \widehat{v}_{i,\ell_{i,j}(2)}(\xi,x)* \cdots 
	 		   *\partial_{x_j}\widehat{v}_{i,\ell_{i,j}(A_{i,j})}(\xi,x).
    	\end{aligned}
    \end{equation*} 
    Let $\sum_{i=1}^{m}\sum_{s=1}^{\alpha_i}\ell_i(s)=\ell'$. 
    Then we have from Lemma \ref{conv} and corollary  \ref{Cor-1}  
     \begin{equation*}
    	\begin{aligned}
    	 I\ll &\frac{|\xi|^{\ell'-k}e^{c|\xi|^{k}}}{\Gamma(\ell'/k)}
    	 \big(\prod_{1\leq i\leq m} M_{i,\ell_i(1)} M_{i,\ell_i(2)}\cdots  M_{i,\ell_i(\alpha_i)}\big)\\ &
    	 \big(\prod_{1\leq i\leq m}
    	    		\frac{\Theta^{(\ell_i(1))}(X)\Theta^{(\ell_i(2))}(X)\cdots \Theta^{(\ell_i(\alpha_i))}(X)}
    	    		{\ell_i(1)!\ell_i(2)!\cdots \ell_i(\alpha_i)!}\big) \\
    		   \ll & \big(\prod_{1\leq i\leq m}
			    	M_{i,\ell_i(1)} M_{i,\ell_i(2)}\cdots  M_{i,\ell_i(\alpha_i)}\big)
    			 \frac{|\xi|^{\ell'-k}e^{c|\xi|^{k}}}{\Gamma(\ell'/k)}
    	  	\frac{\Theta^{(\ell')}(X)}{\ell'!}.
    	\end{aligned}
    \end{equation*} 
      Let $\sum_{\begin{subarray}\ (i,j)\in \Delta \end{subarray}}
      \sum_{s=1}^{A_{i,j}}\ell_{i,j}(s)=\ell''$. We have also 
          \begin{equation*}
    	\begin{aligned}
    	 II \ll &\frac{|\xi|^{\ell''-k}e^{c|\xi|^{k}}}{\Gamma(\ell''/k)}\big(\prod_{i,j\in \Delta}
    	 M_{i,\ell_{i,j}(1)} M_{i,\ell_{i,j}(2)}\cdots  M_{i,\ell_{i,j}(A_{i,j})}\big)\\ &
			\big(	\prod_{i,j\in \Delta}\frac{\Theta^{(\ell_{i,j}(1)+1)}(X)\Theta^{(\ell_{i,j}(2)+1)}(X)\cdots 
				\Theta^{(\ell_{i,j}(A_{i,j})+1)}(X)}
    	    		{\ell_{i,j}(1)!\ell_{i,j}(2)!\cdots \ell_{i,j}(A_{i,j})!}\big) \\
    	   \ll & \big(\prod_{i,j\in \Delta}
	    	M_{i,\ell_{i,j}(1)} M_{i,\ell_{i,j}(2)}\cdots  M_{i,\ell_{i,j}(A_{i,j})}\big) \frac{|\xi|^{\ell''-k}e^{c|\xi|^{k}}}{\Gamma(\ell''/k)} 
		\frac{\Theta^{(\ell''+1)}(X)}{R^{|A|-1}\ell''!}.
    	\end{aligned}
    \end{equation*}
	    Therefore
    \begin{equation*}
    	\begin{aligned}\ & 	I*II \ll\big(\prod_{1\leq i\leq m} 
    	M_{i,\ell_i(1)} M_{i,\ell_i(2)}\cdots  M_{i,\ell_i(\alpha_i)}\big) \\ & \times
    \big(	\prod_{i,j\in \Delta}
    	M_{i,\ell_{i,j}(1)} M_{i,\ell_{i,j}(2)}\cdots  M_{i,\ell_{i,j}(A_{i,j})}\big) 
    	 \frac{|\xi|^{\ell'+\ell''-k}e^{c|\xi|^{k}}}{\Gamma(\frac{\ell'+\ell''}{k})}
    	 \frac{\Theta^{(\ell')}(X)}{\ell'!} \frac{\Theta^{(\ell''+1)}(X)}{R^{|A|-1}\ell''!}.
    	\end{aligned}
    \end{equation*}
	     It follows from \eqref{eq3.7} that $\frac{\Theta^{(\ell')}(X)\Theta^{(\ell''+1)}(X)}{\ell'! \ell''!}\ll 
      \frac{\Theta^{(\ell'+\ell''+1)}(X)}{(\ell'+\ell'')!} $ and
    \begin{equation*}
     	\begin{aligned}\ 
		    & \frac{\Theta^{(\ell'+\ell''+1)}(X)}{(\ell'+\ell'')!}
	    	\frac{\Theta^{(|\alpha|+|A|-1)}(X)}{(|\alpha|+|A|-1)!} \ll
	    	\frac{\Theta^{(\ell'+\ell''+|\alpha|+|A|)}(X)}{(\ell'+\ell''+|\alpha|+|A|-1)!}.
    	\end{aligned}
    \end{equation*}
    Hence we get from   \eqref{g^} 
    \begin{equation}
    	\begin{aligned}\ &
    	\widehat{g}_{r,\alpha,A}(\xi,x)*I*II \ll \frac{GC_0^{|\alpha|+|A|}}{R^{|A|-1}}
	    	\big(\prod_{1\leq i\leq m} 
	    	M_{i,\ell_i(1)} M_{i,\ell_i(2)}\cdots  M_{i,\ell_i(\alpha_i)}\big) \\ & \times\big(
	    	\prod_{(i,j)\in \Delta}(M_{i,\ell_{i,j}(1)} M_{i,\ell_{i,j}(2)}\cdots  M_{i,\ell_{i,j}(A_{i,j})}\big) 
	    	 \\ &  
	    		\times 	 \frac{|\xi|^{\ell'+\ell''+|\alpha|+|A|-k}e^{c|\xi|^{k}}}
	    		{\Gamma(\frac{\ell'+\ell''+|\alpha|+|A|}{k})}
	    	\frac{\Theta^{(\ell'+\ell''+|\alpha|+|A|)}(X)}{(\ell'+\ell''+|\alpha|+|A|-1)!}.
	    	\end{aligned}
    \end{equation}
     Let  
  	\begin{equation}
			\begin{aligned}
				{\Sigma}(L)&=\Big\{\big(\alpha,A,\ell_i(s),\ell_{i,j}(s),
				\; 1\leq i \leq m,\;1 \leq j \leq n\big); \\
		&\sum_{i=1}^m\sum_{s=1}^{\alpha_i}\ell_i(s)+
				\sum_{(i,j)\in \Delta}
				\sum_{s=1}^{A_{i,j}}\ell_{i,j}(s)+|\alpha|+|A|=L \Big\}.
			\end{aligned} \label{Sigma_ell}
		\end{equation}
     Then 
		    \begin{equation*}
    	\begin{aligned}\ &{\mathcal G}_{r,L}(\xi,x)
    	\ll\bigg[ \sum_{\Sigma(L)}
    	\frac{GC_0^{|\alpha|+|A|}}{R^{|A|-1}}\prod_{1\leq i\leq m}\big( 
    	M_{i,\ell_i(1)} M_{i,\ell_i(2)}\cdots  M_{i,\ell_i(\alpha_i)}\big) \\ & 
	    	\prod_{(i,j)\in \Delta}\big (   	
    	M_{i,\ell_{i,j}(1)} M_{i,\ell_{i,j}(2)}\cdots  M_{i,\ell_{i,j}(A_{i,j})}\big) \bigg] 	
    	 \frac{|\xi|^{\L-k}e^{c|\xi|^{k}}}{\Gamma(\frac{L}{k})}
    	\frac{\Theta^{(L)}(X)}{(L-1) !}.
	    	\end{aligned}
    \end{equation*}
    Let 
    \begin{equation}
    	\begin{aligned}\ & M_{r,L} =\sum_{\Sigma(L)}
 		   	G_{\alpha,A}'\prod_{1\leq i\leq m} \big(
    		M_{i,\ell_i(1)} M_{i,\ell_i(2)}\cdots  M_{i,\ell_i(\alpha_i)}\big) \\ 
    	& \times \prod_{(i,j)\in \Delta}\big(M_{i,\ell_{i,j}(1)} M_{i,\ell_{i,j}(2)}\cdots  
    		M_{i,\ell_{i,j}(A_{i,j})}\big)\quad  	G_{\alpha,A}'=\frac{GC_0^{|\alpha|+|A|}}{R^{|A|-1}}.
    	\end{aligned} \label{Mr'}
    \end{equation}
    Then
    \begin{equation}
    	\begin{aligned}
    	\ &{\mathcal G}_{r,L}(\xi,x) \ll M_{r,L}
    	  \frac{|\xi|^{L-k}e^{c|\xi|^{k}}}{\Gamma(\frac{L}{k})}
    	\frac{\Theta^{(L)}(X)}{(L-1) !}.
    	\end{aligned}
    \end{equation}
    $\widehat{v}_{r,L}(\xi,x)$ is defined by 
    	\begin{equation*}\left \{
		\begin{aligned}\
		  & (\xi \partial_{\xi}+k+1)\widehat{v}_{r,L}(\xi,x)={\mathcal G}_{r,L}(\xi,x)\;\; L\geq 2.\\
		  & \widehat{v}_{r,L}(\xi,x)=\xi^{-k-1} 
		  \int_0^{\xi}\eta^{k} {\mathcal G}_{r,L}(\eta,x)d\eta. 
		\end{aligned} .\right .
	\end{equation*}
			It follows from Lemma \ref {lem2.6} that   
 	     $\widehat{v}_{r,L}(\xi,x)\ll M_{r,L}
    	 \frac{|\xi|^{L-k}e^{c|\xi|^{k}}}{\Gamma(\frac{L}{k})}
    	\frac{\Theta^{(L)}(X)}{L !}$ . Thus we get $M_{r,L}$.
   \subsubsection{\normalsize  Proof of Proposition \ref{prop-main} -- convergence}
    	  We show convergence of $\sum_{\ell=1}^{\infty}M_{r,\ell}s^{\ell}$.
				First we sum up  notations and definitions again. 
				Let $(Y,z)=(y_1,y_2,\dots,y_m,z)$, $Y^{\alpha}=\prod_{i=1}^{m}y_i^{\alpha_i}$ for 
				$\alpha=(\alpha_1,\dots,\alpha_m)\in \mathbb{N}^m$,
				$Y^{A}=\prod_{i=1}^{m}\big(\prod_{j=1}^{n}y_i^{A_{i,j}}\big)$ for $A=(A_{i,j};(i,j)\in \Delta) \in
				\mathbb{N}^{mn}$. Let
		\begin{equation}
			\begin{aligned} G_{r}(Y,z)&=\sum_{ (\alpha, A); 
			|\alpha|+|A|\geq 1 }
				 z^{|\alpha|+|A|} {G_{\alpha,A}'}Y^{\alpha}Y^A + zM_{r,1}
			\end{aligned}
		\end{equation}
		(see \eqref{Mr'}).
		$G_{r}(Y,z)$ is holomorphic in a neighborhood of $(Y,z)=(0,0)$,\; $G_{r}(Y,0)=0$. 
		We use the method of implicit functions used in Gerard, Tahara \cite{GT} 
		and others in order to show convergence 
		of $\sum_{\ell=1}^{\infty}M_{r,\ell}s^{\ell}$. Consider the following system of functional equations
	\begin{equation}
		\begin{aligned}
		 y_r=G_r(Y,z), \; r=1,2, \cdots m.
		\end{aligned}\label{fe}
	\end{equation}
		It follows from the implicit function theorem that there exists
		a unique holomorphic solution $Y(z)=(y_1(z), \cdots,y_m(z))$ of \eqref{fe} with $Y(0)=0$. Let 
		$y_i(z)=\sum_{\ell_i=1}^{\infty}C_{i,\ell_i}z^{\ell_i}$. Then we have 
	\begin{equation*}
		\begin{aligned}\ 
			y_i^{\alpha_i}& =(\sum_{\ell_i=1}^{\infty}C_{i,\ell_i}z^{\ell_i})^{\alpha_i}=\prod_{s=1}^{\alpha_i}
		  \big(\sum_{\ell_i(s)=1}^{\infty}C_{i,\ell_i(s)}z^{\ell_i(s)}\big)\\ &
			=\sum_{\ell=\alpha_i}^{\infty}z^{\ell}\big(\sum_{\sum_{s=1}^{\alpha_i}\ell_i(s)=\ell}
			   \prod_{s=1}^{\alpha_i}C_{i,\ell_i(s)}\big), \\
				Y^{\alpha} &=\prod_{i=1}^{m}y_i^{\alpha_i} 
			 =\sum_{\ell'=|\alpha|}^{\infty}z^{\ell'}\big(\sum_{
			 \sum_{i=1}^{m}\sum_{s=1}^{\alpha_i}\ell_i(s)=\ell'}
			   \prod_{i=1}^{m}\prod_{s=1}^{\alpha_i}C_{i,\ell_i(s)}\big),
		\end{aligned}
	\end{equation*}	
	\begin{equation*}
		\begin{aligned}\ 
			Y^{A}=\prod_{(i,j)\in \Delta}y_{i}^{A_{i,j}}
			=\sum_{\ell''=|A|}^{\infty}z^{\ell''}\big(\sum_{\begin{subarray}\ 
			\sum_{ (i,j) \in \Delta } \sum_{s=1}^{A_{i,j}}\ell_{i,j}(s)=\ell'' \end{subarray}}
		  \prod_{(i,j)\in \Delta} \prod_{s=1}^{A_{i,j}}C_{i,\ell_{i,j}(s)}\big) 
		\end{aligned}
	\end{equation*}	
		and 
	\begin{equation*}
		\begin{aligned}\  
				& Y^{\alpha}Y^{A}=\Big(\sum_{\ell'=|\alpha|}^{\infty}z^{\ell'}\big(\sum_{
					 \sum_{i=1}^{m}\sum_{s=1}^{\alpha_i}\ell_i(s)=\ell'}\prod_{i=1}^{m}
				   \prod_{s=1}^{\alpha_i}C_{i,\ell_i(s)} \big)\Big) \\
			  &	\times \Big(  \sum_{\ell''=|A|}^{\infty}z^{\ell''}\big(\sum_{\sum_
				 {(i,j)\in \Delta}\sum_{s=1}^{A_{i,j}}\ell_{i,j}(s)=\ell''} \prod_{(i,j)\in \Delta}
			   \prod_{s=1}^{A_{i,j}}C_{i,\ell_{i,j}(s)}\big)\Big) \\
			  & = \sum_{\ell=|\alpha|+|A|}^{\infty}z^{\ell}\Big(\sum_{\begin{subarray}\
					 \sum_{i=1}^{m}\sum_{s=1}^{\alpha_i}\ell_i(s)+ \\
					 \sum_{(i',j)\in \Delta}\sum_{s=1}^{A_{i',j}}\ell_{i',j}(s)=\ell  \end{subarray}}
			   \prod_{i=1}^{m}\prod_{s=1}^{\alpha_i}C_{i,\ell_i(s)}  
			  \prod_{(i',j)\in \Delta}	 \prod_{s=1}^{A_{i',j}}C_{i',\ell_{i',j}(s)}\Big).
		\end{aligned}
	\end{equation*}	
 		We have for $|\alpha|+|A| \geq 1$
	\begin{equation*}
		\begin{aligned}\ &
				z^{|\alpha|+|A|}G_{\alpha,A}'Y^{\alpha}Y^A   =\\
	   &  G_{\alpha,A}'	\sum_{\ell=2(|\alpha|+|A|)}^{\infty}
	   z^{\ell}\Big(\sum_{\begin{subarray}\; 
			 \sum_{i=1}^{m}\sum_{s=1}^{\alpha_i}\ell_i(s)+ \\
			 \sum_{(i',j)\in \Delta}\sum_{s=1}^{A_{i,j}}\ell_{i,j}(s)+|\alpha|+|A|=\ell  \end{subarray}} 
			 \big(   \prod_{i=1}^{m}\prod_{s=1}^{\alpha_i}C_{i,\ell_i(s)}  
			  \prod_{(i,j)\in \Delta}	 \prod_{s=1}^{A_{i,j}}C_{i,\ell_{i,j}(s)}\Big).
		\end{aligned}
	\end{equation*}
	We get from 
	 $y_r=\sum_{(\alpha,A);	|\alpha|+|A|\geq 1}z^{|\alpha|+|A|}G'_{\alpha,A}Y^{\alpha}Y^A+zM_{r,1} $ 
	\begin{equation}
		\begin{aligned}\ &
		y_r=\sum_{\ell=1}^{\infty}	C_{r,\ell}z^{\ell}
		=	\sum_{|\alpha|+|A|\geq 1}	z^{|\alpha|+|A|}G_{\alpha,A}'Y^{\alpha}Y^A +zM_{r,1}\\ &
		 =\sum_{\ell=2}^{\infty}z^{\ell} \Big({\sum}_{\ell} 
		 G_{\alpha,A}' \prod_{i=1}^{m}\prod_{s=1}^{\alpha_i}C_{i,\ell_i(s)}  
		  \prod_{(i',j)\in \Delta}	 \prod_{s=1}^{A_{i',j}}C_{i',\ell_{i',j}(s)}\Big)+zM_{r,1},
			\end{aligned}
	\end{equation}
		where
	\begin{equation}
		\begin{aligned}\
			{\sum}_{\ell}=\big\{ &(\alpha,A)\in \mathbb{N}^{m}\times \mathbb{N}^{mn},\; 
			(\ell_{i}(1),\ell_{i}(2),\cdots,\ell_{i}(\alpha_{i}))_{1\leq i \leq m}, \\ &
			((\ell_{i',j}(1),\ell_{i',j}(2),\cdots,\ell_{i',j}(A_{i',j}))_{1\leq i' \leq m\; 1\leq j \leq n}; \\ & 
	 		\sum_{i=1}^m\sum_{s=1}^{\alpha_i}\ell_i(s)+\sum_{\begin{subarray}\ 1\leq i' \leq m \\
			1 \leq j \leq n \end{subarray}}\sum_{s=1}^{A_{i',j}}\ell_{i',j}(s)+|\alpha|+|A|=\ell \big\}.
						\end{aligned}
		\end{equation}
		Hence ${\Sigma}_{\ell}=\Sigma(\ell)$ (\eqref{Sigma_ell}) holds.
			We have $C_{r,1}=M_{1,\ell}$ and $C_{r,\ell}=M_{r,\ell}$ from \eqref{Mr'} for $\ell \geq 2$ and 
			$\sum_{\ell=1}^{\infty}M_{r,\ell}s^{\ell}$ converges. Thus Proposition \ref{prop-main}-(1) is shown. 
			\par 
     Since $\sum_{\ell=1}^{\infty}M_{r,\ell}s^{\ell}$ converges, there exist $A,C,R_0>0$ 
			such that $ | M_{r,\ell}\frac{\Theta^{(\ell)}(X)}{\ell!}|\leq AC^{\ell}$ for $|x|<R_0$. 
			Take $\varepsilon=1$ and let 
			$\widehat{v}_r(\xi,x)=\sum_{\ell=1}^{\infty}	\widehat{v}_r^\ell(\xi,x)$.
			Then there is $c'>0$ such that
				\begin{equation*}
    		\begin{aligned}
    		\sum_{\ell=1}^{\infty} |\widehat{v}_r^\ell(\xi,x)|\leq A \sum_{\ell=1}^{\infty}C^{\ell}
    		 \frac{|\xi|^{	\ell-k}e^{c|\xi|^{k}}}{\Gamma(\ell/k)}\leq  A
    		 \frac{|\xi|^{1-k}e^{c'|\xi|^{k}}}{\Gamma(1/k)}.
    		\end{aligned}
    	\end{equation*}
    	     Thus we get Proposition \ref{prop-main}   \qed
	\section{\normalsize Multisummable functions}  
	 			More generally we study Cauchy problem in spaces of multisummable functions.
	 			We give their definition and elementary properties (for the detail \cite{Bal}).
	   \subsection{\normalsize	Acceleration operator}				 
		 		We define $(k',k)$-acceleration operator in a direction $\theta$ due to \'{E}calle. It is denoted 
				by $\mathscr{A}_{k',k,\theta}$. Let $0<k<k'$ and $1/\kappa=1/k-1/k'$.  $\mathscr{A}_{k',k,\theta}$ is
				defined by 
			\begin{equation}
				\begin{aligned}
				  \mathscr{A}_{k',k,\theta}:=\mathscr{B}_{k',\theta}\mathscr{L}_{k,\theta}.
				\end{aligned}
			\end{equation} 
				Let $\widehat{I}=(\theta-\varepsilon,\theta+\varepsilon)\  (\varepsilon>0)$,   
				$I=(\theta-\pi/2\kappa-\varepsilon,\theta+\pi/2\kappa+\varepsilon)$ and $U \subset \mathbb{C}^n$ 
				be a domain.  $\mathscr{A}_{k',k,\theta}$ is defined for $\psi(\xi,x) \in {\rm Exp}(k, S(\widehat{I})
				\times U)$ with
		\begin{equation*}
			\begin{aligned}
				|\psi(\xi,x)|\leq M|\xi|^{s-k}\;( s>0)\;\; \mbox\it {in} \; \{\xi \in S(\widehat{I}); |\xi|\leq 1\}.	
			\end{aligned}
		\end{equation*}			
		  We can extend it to ${\rm Exp}(\kappa, S(\widehat{I}\times U)$
				and  $(\mathscr{A}_{k',k,\theta}\phi)(\xi,x) \in {\mathscr O}(S_0(I)\times U)$. 
				Then  the following lemma holds.
		\begin{lem}
			  Let $\phi_i(\xi,x)\in {\rm Exp}(\kappa, S(\widehat{I})\times U)$ $(i=0,1,2)$ with 
			  $|\phi_i(\xi,x)|\leq C|\xi|^{s-k}\; (s>0)$ in $\{\xi\in S(\widehat{I}); |\xi|<1\}$. 
			  Then 
		  	\begin{equation}
					  (	\mathscr{A}_{k',k,\theta}\phi_1)\underset{k'}{*}(\mathscr{A}_{k',k,\theta}\phi_2)
						=\mathscr{A}_{k',k,\theta}(\phi_1 \underset{k}{*}\phi_2 ) \label{eq4.2}
				\end{equation}  
				and				
				\begin{equation}
						\begin{aligned}
						 &  (\xi \partial_{\xi}+k'+1)\mathscr{A}_{k',k,\theta}\phi_0
						 =\mathscr{A}_{k',k,\theta}(\xi \partial_{\xi}+k+1)\phi_0			 
						\end{aligned} \label{eq4.3}
					\end{equation}  hold.\label{lem4.1}
			\end{lem}
			\begin{proof}
			   We show \eqref{eq4.3}. 
			    Assume $\phi_0(\xi,x) \in {\rm Exp}(k,S(\widehat{I})\times U)$ and let 
			    $w(t,x)=\big(\mathscr{L}_{k,\theta}\phi_0\big)(t,x)$. Then  
					\begin{equation*}
						\begin{aligned} &
 						 (tw(t,x))'=\int_0^{\infty}e^{-(\frac{\xi}{t})^k}(\xi\partial_{\xi}+k+1)\phi_0(\xi,x)d\xi^k=
 						 \mathscr{L}_{k,\theta}(\xi\partial_{\xi}+k+1)\phi_0
 						 ,  \\
					 & \mathscr{B}_{k',\theta}(tw(t,x))'=\frac{1}{2\pi i}\int_{\mathcal{C}}
					 e^{(\frac{\xi}{t})^{k'}}(tw(t,x))'dt^{-k'}	\\
					 &=\frac{1}{2\pi i}\int_{\mathcal{C}}\big(k'+1+
					 k'\xi^{k'}t^{-k'}\big)e^{(\frac{\xi}{t})^{k'}}w(t,x)dt^{-k'} 
					  =(k'+1 +\xi \partial_{\xi})\mathscr{B}_{k',\theta}w \\
					  & =(k'+1 +\xi \partial_{\xi})\mathscr{B}_{k',\theta}\mathscr{L}_{k,\theta}\phi_0
					  =(k'+1 +\xi \partial_{\xi})\mathscr{A}_{k',k,\theta}\phi_0.
						\end{aligned}
					\end{equation*}
					This means \eqref{eq4.3} holds for  $\phi_0(\xi,x) \in {\rm Exp}(k,S(\widehat{I})\times U)$. 
			    If $\phi_0(\xi,x) \in {\rm Exp}(\kappa,S(\widehat{I})\times U)$, 
			    equality \eqref{eq4.3} follows from that
			    $\mathscr{A}_{k',k,\theta}$ can be extensible to ${\rm Exp}(\kappa,S(\widehat{I})\times U)$.
			     As for \eqref{eq4.2} it is known (see \cite{Bal}).
			\end{proof}	
		\subsection{\normalsize  {\bf k}-Multisummable functions}  
	 	    	Let $\mbox{\bf k}=(k_1,k_2,\cdots,k_p)$,\;$0<k_1<k_2<\cdots<k_p<k_{p+1}=+\infty$. 
 					Let us define {\bf k}-multisummability of $\widetilde{f}(t,Y)\in {\mathscr O}(\Omega)[[t]]$. 
 					Let $\kappa_i^{-1}=k_{i}^{-1}-k_{i+1}^{-1}$ for $1\leq i \leq p $. 
			   	Let ${\mathbf \theta}=(\theta_1,\theta_{2}, \cdots ,\theta_p)$ with 
					$|\theta_{i}-\theta_{i+1}|\leq  \frac{\pi}{2\kappa_i}$ for $1\leq i \leq p-1$
				  if $p\geq 2$. ${\mathbf \theta}$ is called multidirection.
				Let $\widetilde{u}(t,Y)=\sum_{n=1}^{\infty}u_n(Y) t^n \in \mathscr{O}(\Omega)[[t]]$ with
				$u_0(Y)=0$. The formal
				 $k$-Borel transform $\widetilde{\mathscr{B}}_k \widetilde{u}$ is defined by
		 \begin{equation}
			 	\begin{aligned}
			 		(\widetilde{\mathscr{B}}_k \widetilde{u})(\xi,Y)=
			 		\sum_{n=1}^{\infty}\frac{u_n(Y) \xi^{n-k}}{\Gamma(\frac{n}{k})}.
			 	\end{aligned}
		 \end{equation} 
	       Let $S_i^*=S(\theta_i-\varepsilon_i, \theta_i+\varepsilon_i)\;( \varepsilon_i>0)$.	
			\begin{defn} 
				\begin{enumerate}
			 		\item[{\rm (I)}] Let $\widetilde{f}(t,Y)\in {\mathscr O}(\Omega)[[t]]$ with $\widetilde{f}(0,Y)=0$.
			 		Then $\widetilde{f}(t,Y)$  is said to be ${\bf k}$-multisummable in a multidirection 
			 		${\mathbf \theta}$ in $\Omega$, if there exist $S_i^*\ (1\leq i \leq p)$ and
			 		the followings hold:
   			 \begin{enumerate}
			 		\item[{\rm (I-1)}] Let $f^{1}(\xi,Y)=\big(\widetilde{{\mathscr B}}_{k_1}\widehat{f}\big)(\xi,Y)$.
					 	Then $\xi^{k_1-1}f^{1}(\xi,Y)$ is holomorphic in $\{|\xi|<r_0\}\times \Omega$ and 
					 	$f^{1}(\xi,Y)$ is 
					 	holomorphically extensible to $S^*_1\times\Omega$ such that 
					 	$f^{1}(\xi,Y)\in {\rm Exp}(\kappa_1,S^*_1\times\Omega)$.
			 		\item[\rm(I-2)] Let $j \in \{2,\cdots,p\}$. $f^{j-1}(\xi,Y)\in {\rm Exp}
			 		(\kappa_{j-1},S^*_{j-1}\times\Omega)$ and    
					$f^{j}(\xi,Y)=(\mathscr{A}_{k_{j},k_{j-1},\theta_{j-1}}f^{j-1})(\xi,Y)$. Then   
					$f^{j}(\xi,Y)\in {\rm Exp}(\kappa_{j},S^*_{j}\times\Omega)$.		 		
			 \end{enumerate} 
			 	  {\bf k}-sum of $\widetilde{f}(t,Y)$ in multidirection ${\mathbf \theta}$ is defined
	   		   by $({\mathscr L}_{k_p,\theta_p}f^{p})(t,Y)$ and denoted simply by $f(t,Y)$.
					$f(t,Y)$ is also said to be ${\bf k}$-multisummable.  
			 \item[{\rm (II)}] Suppose $\widetilde{f}(0,Y)\not =0$. Let 
			 $\widetilde{f}_*(t,Y)=\widetilde{f}(t,Y)-\widetilde{f}(0,Y)$. If $\widetilde{f}_*(t,Y)$ is 
			  ${\bf k}$-multisummable in a multidirection ${\mathbf \theta}$,  then
			$\widetilde{f}(t,Y)$ is said to be ${\bf k}$-multisummable.   {\bf k}-sum of $\widetilde{f}(t,Y)$ is defined
	   		by $\widetilde{f}(0,Y)+({\mathscr L}_{k_p,\theta_p}f^{p}_*)(t,Y).$ 
	   		 \end{enumerate}  
	   		\end{defn}
			We have	$f(t,Y) \in {\mathscr O}(S_{p,0 }\times \Omega)$, 
				$S_p=S(\theta_p-\pi/2k_p-\varepsilon_p, \theta_p+\pi/2k_p+\varepsilon_p) $.
	   	 \begin{rem} 
	   		  \begin{enumerate}
			 	\item[{\rm(1)}] Let $\widetilde{f}(t,Y)$ be ${\bf k}$-multisummable. Then
							$\widetilde{f}(t,Y)\in 
			 				{\mathscr O}(\Omega)[[t]]_{\frac{1}{k_1}}$. 
			 	\item[{\rm(2)}]	 Let $\widetilde{f}(t,Y)$ be {\bf k}-multisummable with $\widetilde{f}(0,Y)=0$.  
				We denote $(\widetilde{{\mathscr B}}_{k_1}\widetilde{f})(\xi,Y)$ by $\widehat{f}(\xi,Y)$ and 
						\begin{equation}
								\begin{aligned} &
						  			\mathscr{M}_{{\bf k},{\bf \theta}}:=
									\mathscr{L}_{k_p,\theta_p}\mathscr{A}_{k_p,k_{k_{p-1}},\theta_{p-1}}\cdots
									\mathscr{A}_{k_{2},k_1,\theta_1}.
							\end{aligned} 
						\end{equation}
						Then $f(t,Y)=(\mathscr{M}_{{\bf k},{\bf \theta}}\widehat{f})(t,Y)$.
			 			 We may say that 
			 			 $f(t,Y)$ is represented by repeated (multi) Laplace Borel transform. More precisely 
			 	\begin{equation*}
						\begin{aligned}\ 
					  & f^1(\xi,Y)=\widehat{f}(\xi,Y)
						& f^{j}(\xi,Y)=(\mathscr{A}_{k_{j},k_{j-1},\theta_{j-1}}f^{j-1})(\xi,Y) \quad    2 \leq  j\leq p 
					\end{aligned}
				\end{equation*}
				  and	$f(t,Y)=(\mathscr{L}_{k_p,\theta_p}f^{p})(t,Y)$.
			 	 \end{enumerate}   
	   		 \end{rem}
			 		Let $\widetilde{w}(t)\in \mathbb{C}[[t]]$ with $\widetilde{w}(0)=0$ be ${\bf k}$-multisummable,
			 		 $\widehat{w}=\widetilde{{\mathscr B}}_{k_1}\widetilde{w}$ and $w(t)$ be {\bf k}-sum of 
			 		 $\widetilde{w}(t)$. Then from  Lemma \ref{lem4.1} we have an important relation
				\begin{equation}
				\begin{aligned}\  
				   & (t w(t))'=\mathscr{L}_{k_p,\theta_p}(\xi \partial_{\xi}+k_{p}+1)w^{p}\\
						&=\mathscr{L}_{k_p,\theta_p} \mathscr{A}_{k_p,k_{p-1},\theta_{p-1}}(\xi \partial_{\xi}+k_{p-1}+1)
						w^{p-1} \\
						&=\mathscr{L}_{k_p,\theta_p} \mathscr{A}_{k_p,k_{p-1},\theta_{p-1}}\cdots
							\mathscr{A}_{k_{j+1},k_{j},\theta_{j}}
							(\xi \partial_{\xi}+k_{j}+1){w}^{j} \\	
						&=\mathscr{L}_{k_p,\theta_p} \mathscr{A}_{k_p,k_{p-1},\theta_{p-1}}
							\mathscr{A}_{k_{p-1},k_{p-2},\theta_{p-2}} \cdots	\cdots
							\mathscr{A}_{k_{2},k_1,\theta_1}(\xi \partial_{\xi}+k_{1}+1)\widehat{w}. 	
				\end{aligned}\label{w}
			\end{equation} 
			\section{\normalsize Cauchy problem in ${\bf k}$-multisummable function spaces}
				 We study (CP) (\eqref{CP}) in spaces of multisummable functions. 
				 Let $\Omega_0=\{x\in \mathbb{C}^n; |x|<R_0\}$,  
			  $\Omega'=\{(U,P)\in \mathbb{C}^n\times \mathbb{C}^{mn}; |U|, |P|<R_1 \}$ and 
			  $\Omega=\Omega_0\times\Omega'$. In this section we assume that $\{f_i(t,x,U,P)\}_{i=1}^m$ in (CP)
			  are ${\bf k}$-multisummable in a multidirection ${\bf \theta}$ in $\Omega$.	
			  We reduce (CP) so that $u_{i}(0,x)=\partial_{t}u_{i}(0,x)=0$ for $1\leq i \leq m$ and 
			  let $u_i(t,x)=tv_i(t,x)$ and $g_i(t,x, V, P):=f(t,x, tV, tP)$. Then
			    we can transform (CP) to the following system of equations as before (see \eqref {CP-1}),
			\begin{equation}
				\begin{aligned} \
				 & \partial_t (t v_i)=g_i(t,x, V, \nabla_xV ) \\
				 & =\sum_{\ell=0}^{\infty}\big(\sum_{\begin{subarray}\ (\alpha,A)\\ |\alpha|+|A|=\ell
					  \end{subarray}}g_{i,\alpha,A}(t,x) 
				  	V^{\alpha} (\nabla_x V^{A}\big) 
				  	\quad 1\leq i \leq m, \\
				 & V(t,x)=(v_1(t,x),\dots,v_m(t,x)) \quad  V(0,x)=0 	
				\end{aligned}  \label{CP-M}
			\end{equation}
				with $ g_{i,0,0}(0,x)=0 $. 
			\begin{lem}
         $\{g_i(t,x, V, P)\}_{i=1}^m$ are ${\bf k}$-multisummable in multidirection ${\bf \theta}$. 
         \label{lem5.1}
       \end{lem}
       The proof is given in Appendix.        
			  	Let $\widetilde{V}(t,x)=(\widetilde{v}_1(t,x),\dots,\widetilde{v}_m(t,x)), 
				\widetilde{v}_i(t,x) \in \mathscr{O}(\Omega_0)[[t]]$ be a unique formal power series solution of
				\eqref{CP-M}. Then
			\begin{thm}
					      There exists  $\omega_0=\{x\in {\mathbb C}^n; |x|<r\}$ such that 
						   $\{\widetilde{v}_i(t,x)\}_{i=1}^m$ are {\bf k}-multisummable  in the multidirection 
						   ${\bf \theta}$ in $\omega_0$. 	
										\label{th-2}
			\end{thm} 
			   Let us prepare to show Theorem \ref{th-2}.  Let   ${g}_i(t,x,V,P)$ be $\bf k$-sum of
			   $\widetilde{g}_i(t,x,V,P)\in \mathscr{O}(\Omega)[[t]]$, 
			\begin{equation*} \left \{
					\begin{aligned}\ 
					  & g^1_i(\xi,x,V,P)=\widehat{g}_i(\xi,x,V,P), \;\;  
					  \widehat{g}_i(\xi,x, V,P)=(\widetilde{\mathscr{B}}\widetilde{g}_i)(\xi,x, V, P), \\
						& g_i^s(\xi,x, V, P)=(\mathscr{A}_{k_{s},k_{s-1},\theta_{s-1}}g^{s-1}_i)(\xi,x,V,P) 
						\quad    2 \leq  s \leq p, \\
					  & g_i(t,x,V,P)=(\mathscr{L}_{k_p,\theta_p}g^p_i)(t,x,V,P)=
					  (\mathscr{M}_{{\bf k},{\bf \theta}} \widehat{g}_i)(t,x, V, P)
				\end{aligned}\right.
			\end{equation*}	
					and 	 
			\begin{equation*}
				\begin{aligned}&
				 g_i^s(\xi,x, V, P)=\sum_{\begin{subarray}\ (\alpha,A) \end{subarray}}g_{i,\alpha,A}^s(\xi,x)
				  	V^{\alpha}P^{A}\quad 1\leq i \leq m.				 
				\end{aligned}
			\end{equation*}
  			  We will repeat the methods similar to Borel summable case $p$-times for the proof of Theorem 
  			  \ref{th-2}. Firstly let us introduce systems of convolution equations 
  			  {$(C_s)\; (1\leq s \leq p)$} (see \eqref{Ceq}). 
					Let $W(\xi,x)=({w}_1(\xi,x),\dots,{w}_m(\xi,x))$ and    
  		\begin{equation}
				\begin{aligned}   {\rm (C_s)} & \qquad(\xi \partial_{\xi}+k_s+1){w}_r(\xi,x)\\
					 &= \sum_{\ell=0}^{\infty}\big(\sum_{|\alpha|+|A|=\ell}{g}_{r,\alpha,A}^{s}(\xi,x)
					 \underset{k_s}{*} W^{*\alpha}\underset{k_s}{*} 
				  (\nabla_x W)^{*A}\big) 
					\quad 1\leq r \leq m.
				\end{aligned}\label{Cs}
			\end{equation}
			Let $\widetilde{V}(t,x)=(\widetilde{v}_1(t,x),\dots,\widetilde{v}_m(t,x))$, 
			 $\widetilde{v}_i(t,x) \in \mathscr{O}(\Omega_0)[[t]]$, be a formal  
			 solution of \eqref{CP-M}, then 
			$(\widetilde{\mathscr B}_{k_1}\widetilde{V})(\xi,x)\in  {\xi}^{1-k_1}{\mathscr O}(\Omega_0)[[\xi]]^m$ 
			  satisfies $(C_1)$ formally.    
				We will be able to show Theorem \ref{th-2} by solving equations  $(C_1), \dots, (C_p)$ 
				successively. 				
	      The following estimates for $\{{g}^s_{r,\alpha,A}(\xi,x); 1\leq s \leq p\} $ hold.
	      ${\Theta(X)=\theta(X/R)}\;\; 0<R<R_0$, $X=\sum_{i=1}^n x_i$, 
	      $\theta(\tau)$  is introduced in subsection 3.1 Majorant function (\eqref{Theta}).	       
      \begin{prop}
				  There exist positive constants $G_{r,\alpha,A}(s)$ and $c$  such that 
		  	\begin{equation}\left \{
					\begin{aligned}
						 &	{g}^s_{r,\alpha,A}(\xi,x)\ll G_{r,\alpha,A}(s)
						 \frac{\xi^{|\alpha|+|A|-k_s}e^{c|\xi|^{\kappa_s}}}
							{\Gamma(\frac{|\alpha|+|A|}{k_s})}\frac{\Theta^{|\alpha|+|A|-1}(X)}{(|\alpha|+|A|-1)!}
							\quad		|\alpha|+|A|\geq 1, 	\\
							& {g}^s_{r,0,0}(\xi,x)\ll G_{r,0,0}(s)
						 \frac{\xi^{1-k_s}e^{c|\xi|^{\kappa_s}}}{\Gamma(\frac{1}{k_s})}{\Theta(X)},
						  \quad \xi\in S_{s}^*=S(\theta_s-\varepsilon_s, \theta_s+\varepsilon_s)
					\end{aligned}\right . \label{gs-est}
				\end{equation} 
					and $\sum_{\alpha, A}G_{r,\alpha,A}(s)V^{\alpha} P^{A}$ converges in a neighborhood of
					$(V,P)=(0,0)$. \label{prop5.3}
			\end{prop}
				  The proof is given in Appendix. 
	       The following arguments are similar to the preceding sections. 
	       In order to study $(C_s)$ (\eqref{Cs}) we introduce an auxiliary parameter $\varepsilon$ 
	        as Borel summable case  (see \eqref{Conv-e}), 
			\begin{equation}
				\begin{aligned} \ {\rm (C_{s,\varepsilon}}) \quad 
				& (\xi \partial_{\xi}+k_s+1){w}_r(\xi,x,\varepsilon)\\
						 =&\sum_{\ell=1}^{\infty} \varepsilon^{\ell}
						 \big(\sum_{|\alpha|+|A|=\ell}{g}_{r,\alpha,A}^{s}(\xi,x)
						 \underset{k_s}{*}W^{*\alpha}\underset{k_s}{*} 
					  (\nabla_x W)^{*A}\big)+\varepsilon g_{i,0,0}^{s}(\xi,x).
				\end{aligned} \label{C_e'*}
			\end{equation}
			We will solve \eqref{C_e'*} and get a solution $W(\xi,x)$ of \eqref{Cs} by taking $\varepsilon=1$.
	      Let	${w}_{r}(\xi,x,\varepsilon)=\sum_{\ell=1}^{\infty}{w}_{r,\ell}(\xi,x)\varepsilon^{\ell}$\; 
	    $(1\leq r \leq m)$. By the same method in subsection 2.2.1 {\it ``
	    Calculation of $\mathcal{G}_{r,\ell}(\xi,x)$"} (see	\eqref{GL}, \eqref{GL-1} and \eqref{GL-2}),  
	    we have for $|\alpha|+|A|\geq 1$
   	\begin{equation}
			\begin{aligned}
				\ &  
				\varepsilon^{|\alpha|+|A|}{g}_{r,\alpha,A}^s(\xi,x)\underset{k_s}{*} 
				 W^{{*} \alpha}
				\underset{k_s}{*}  (\nabla _x  W)^{* A}=
				\sum_{\ell=2(|\alpha|+|A|)}^{\infty}\varepsilon^{\ell} 
				\Big( \sum_{\begin{subarray}\ 
				\ell'+\ell''+|\alpha|+|A|=\ell\end{subarray}} 
				{g}_{r,\alpha,A}^s(\xi,x)  \\ & \underset{k_s}{*} 
		    \Big ( \sum_{ ({\boldsymbol \alpha},\ell')}
			 	\prod_{1\leq i\leq m}^*\big(w_{i,\ell_i(1)}(\xi,x)
			 	\underset{k_s}{*} w_{i,\ell_{i}(2)}(\xi,x)\underset{k_s}{*}  \cdots 
		 		  \underset{k_s}{*} w_{i,\ell_{i}(\alpha_i)}(\xi,x)\big) \Big)\\
			 	& \underset{k_s}{*} 	\Big (\sum_{ (A,\ell'')} 
			 	\prod_{\begin{subarray}\ (i,j)\in \Delta 	\end{subarray}}^{*}\partial_{x_j}w_{i,\ell_{r,j}(1)}(\xi,x)
			 	\underset{k_s}{*} \partial_{x_j}w_{i,\ell_{i,j}(2)}(\xi,x)\underset{k}{*}  \cdots 
		 		   \underset{k_s}{*} \partial_{x_j}w_{i,\ell_{i,j}(A_{i,j})}(\xi,x)\Big)\Big) \\  
			\end{aligned},
		\end{equation}
	 		where  
		\begin{equation*}
			 	\begin{aligned}\ 
		 		  &  \sum_{ ({\alpha},\ell')}=\{(\ell_{i}(1),\ell_{i}(2),\cdots,\ell_{i}(\alpha_{i}))
			 		  \; 1\leq i \leq m; 
			 			  \sum_{\ 1\leq i \leq m }
			 		\sum_{s=1}^{\alpha_i}\ell_{i}(s)=\ell'\}, \\	  
			 	 		  &  \sum_{ (A,\ell'')}=\{(\ell_{i,j}(1),\ell_{i,j}(2),\cdots,\ell_{i,j}(A_{i,j}))\; (i,j)\in
			 	 		  \Delta ;  \sum_{ (i,j)\in \Delta 	}
			 		\sum_{s=1}^{A_{i,j}}\ell_{i,j}(s)=\ell''\}.
			 	\end{aligned}
		 \end{equation*}
			Then the right hand side of $(C_{s,\varepsilon})$ \eqref{C_e'*} is a power series of $\varepsilon$    
   	\begin{equation}
   		\begin{aligned}
   		 & \sum_{\ell=1}^{\infty}\varepsilon^{\ell}
			 \Big(\sum_{\begin{subarray}\ (\alpha,A); \\ |\alpha|+|A|=\ell \end{subarray}}
			 \widehat{g}_{r,\alpha,A}(\xi,x)\underset{k_s}{*}  
			  {W}^{* \alpha}\underset{k_s}{*}  (\nabla_x {W})^{{*} A}\Big)
			  +	\varepsilon\widehat{g}_{r,0,0}^s(\xi,x) \\
			  &=\sum_{\ell=1}^{\infty}{\mathcal G}_{r,\ell}(s)(\xi,x) \varepsilon^{\ell} 
   		\end{aligned} \label{Conv-e'}
   	\end{equation}
			with   	
		 \begin{equation}	\left \{
			\begin{aligned}\
				&	\mathcal{ G}_{r,1}(s)(\xi,x)=g_{r,0,0}^{s}(\xi,x)\\
				&	\mathcal{ G}_{r,\ell}(s)(\xi,x) \mbox{ is determined by }
				{\{w_{j,q}(\xi,x)\}_{j=1}^m }  \;( 1 \leq q \leq \ell-1). 
			\end{aligned}\right.  \label{G-s}
		\end{equation}
		  We get a system of equations for $\{{w}_{r,\ell}(\xi,x): 1\leq r \leq m, \ell=1,2, \cdots\}$ 
 		  from \eqref{C_e'*} $(C_{s,\varepsilon})$
		\begin{equation}
					\left \{
		 	\begin{aligned}\
			 	& (\xi \partial_{\xi}+k_s+1){w}_{r,1}(\xi,x)=g_{r,0,0}^{s}(\xi,x)\\
			 	& (\xi \partial_{\xi}+k_s+1){w}_{r,\ell}(\xi,x)=\mathcal{G}_{r,\ell}(s)(\xi,x)\quad \ell \geq 2.
	 		\end{aligned}\right . \label{ws}
		\end{equation}
	   	 Thus we have prepared a system of equations $(C_s), (C_{s,\varepsilon})$ and \eqref{ws}
	   	 we need.\par 
		   	Let us constract $V(t,x)=(v_1(t,x),, \dots,v_m(t,x))$ satisfying \eqref{CP-M} such that  
	   	 $V(t,x)=(\mathscr{M}_{{\bf k},{\bf \theta}}\widehat{V})(\xi,x)$,  
	   	 $v_i(t,x)=(\mathscr{M}_{{\bf k},{\bf \theta}}\widehat{v}_i)(t,x)$. 
	   	 Firstly we study $\widehat{V}(\xi,x)$.  We have $\widehat{V}(\xi,x)=(\widetilde{{\mathscr B}}_{k_1}\widetilde{V})(\xi,x)$ 
		formally.
 	       In order to study its analytic property we solve \eqref{ws} for $s=1$, that is,
	   	 $(C_{1,\varepsilon})$,
	\begin{equation}\left \{
		 	\begin{aligned}\
		 	& (\xi \partial_{\xi}+k_1+1){w}_{r,1}(\xi,x)=\widehat{g}_r(\xi,x)\;
		 	(=g_{r,0,0}^{1}(\xi,x)) \\ 
		 	& (\xi \partial_{\xi}+k_1+1){w}_{r,\ell}(\xi,x)=G_{r,\ell}(1)(\xi,x)\quad \ell \geq 2,  \quad
		 	  1\leq r \leq m.
	 	\end{aligned}\right . \label{5.9}
		 \end{equation}
				We can determine 	$\{{w}_{r,\ell}(\xi,x)\}_{r=1}^m (\ell\geq 1)$ and show
				${w}_{r}(\xi,x,\varepsilon)=\sum_{\ell=1}^{\infty}{w}_{r,\ell}(\xi,x)
				\varepsilon^{\ell}$\ ($1\leq r \leq m$) satisfy $(C_{1,\varepsilon})$ 
				as the case of Borel summability (Proposition \ref{prop-main}). Hence we get  
		\begin{prop} {\rm (1)} There exist 
				$\{{w}_{r,\ell}(\xi,x)\}_{r=1}^m (\ell\geq 1)$ satisfying \eqref{5.9} such that 
				$\{\xi^{k_1-1}{w}_{r,\ell}(\xi,x)\}_{r=1}^m$ are  holomorphic in $(\{|\xi|<r\}\cup S^*_1)\times 
				\Omega_0$ for $r>0$. Further there exist constants 
				$\{M_{r,\ell}(1)\}_{r=1}^m$ $(\ell \geq 1)$ and $c$ such that
			\begin{equation}
				\begin{aligned}
					{w}_{r,\ell}(\xi,x)\ll M_{r,\ell}(1)
					 \frac{|\xi|^{\ell-k_1}e^{c|\xi|^{\kappa_1}}}{\Gamma(\ell/k_1)}
					\frac{\Theta^{(\ell)}(X)}{\ell !}
				\end{aligned}
			\end{equation}
				and $\sum_{\ell=1}^{\infty}M_{r,\ell}(1)\tau^{\ell}$ converges in a neighborhood of $\tau=0$. \\
			{\rm (2)}  	${w}_{r}(\xi,x,\varepsilon)=\sum_{\ell=1}^{\infty}w_{r,\ell}(\xi,x)\varepsilon^{\ell}$
				 converge in a neighborhood $\omega$ of $x=0$, and $W(\xi,x,\varepsilon)=({w}_{1}(\xi,x,\varepsilon),\dots, 
		    {w}_{m}(\xi,x,\varepsilon))$ satisfies $(C_{1,\varepsilon})$. Let $\varepsilon=1$ and 
				${w}_{r}(\xi,x)=\sum_{\ell=1}^{\infty}w_{r,\ell}(\xi,x)$. 	Then
			\begin{equation}
				\begin{aligned}
				|{w}_{r}(\xi,x)|\leq M
				{|\xi|^{1-k_1}e^{c'|\xi|^{\kappa_1}}}.
				\end{aligned}\label{eqwr}
			\end{equation}
			and 	$W(\xi,x)=({w}_{1}(\xi,x),\dots,  {w}_{m}(\xi,x))$ satisfies $(C_{1})$. 
				 \label{prop-1}
		\end{prop} 
				We can show Proposition \ref{prop-1} by repeating the same method as the case Borel summability 
		  ({\it the proof of Proposition \ref{prop-main}}). The difference is that 
		  $e^{c|\xi|^{k}}$ in the estimate is replaced by $e^{c|\xi|^{\kappa_1}}$. 
		    Proposition \ref{prop-1} means $W(\xi,x)$
		    has holomorphic in infinite sector $S_1^*$ and $\xi^{k_1-1}W(\xi,x)$
		    is holomorphic at $\xi=0$. Then 
		   \begin{equation}
		   	\begin{aligned}
		   	  w_r(\xi,x)=\sum_{n=1}^{\infty}\frac{{\xi}^{n-k_1}}{\Gamma(\frac{n}{k_1})}w^*_{r,n}(x) \quad
		   	  0<|\xi|<r.
		   	\end{aligned}
		   \end{equation} 
		  Let 
		   $\widetilde{V}(t,x)=({\widetilde v}_1(t,x),\dots,{\widetilde v}_m(t,x))$,  
		   ${\widetilde v}_r(t,x)=\sum_{n=1}^{\infty}v_{r,n}(x)t^n$, be a formal solution and 
		   		   $({\widetilde{\mathscr B}}_{k_1}{\widetilde v}_r)(\xi,x)
		   =\sum_{n=1}^{\infty}\frac{{\xi}^{n-k_1}}{\Gamma(\frac{n}{k_1})}v_{r,n}(x)$. Then  
		   $w^*_{r,n}(x)=v_{r,n}(x)$. 
		 \begin{cor}
		  $(\widetilde{B}_{k_1}\widetilde{V})(\xi,x)=W(\xi,x)$.
		 \end{cor} 
			 Let $V^{1}(\xi,x,\varepsilon):=W(\xi,x,\epsilon)$. 
		       We proceed to next step. It is inductive with respect to $s$. Let $2\leq s \leq p$. 
			 Suppose that $V^{s-1}(\xi,x,\varepsilon)=({v}_1^{s-1}(\xi,x,\varepsilon),\dots,
			 {v}_m^{s-1}(\xi,x,\varepsilon))$ satisfying $(C_{s-1,\epsilon})$ \eqref{C_e'*} is constructed 
			 such that
		 \begin{equation}\left \{
				 	\begin{aligned} &
			 		{v}_r^{s-1}(\xi,x,\varepsilon)=\sum_{\ell=1}^{\infty}{v}_{r,\ell}^{s-1}(\xi,x)\varepsilon^{\ell}\\
				 	& {v}_{r,\ell}^{s-1}(\xi,x)\in {\rm Exp}(\kappa_{s-1},S_{s-1}^*\times\Omega_0)
			 	 \;\; 1\leq r \leq m,\; \ell\geq 1 
			 	\end{aligned}\right. \label{Vs-1}
		 \end{equation}
		   with estimates 
			\begin{equation}
				\begin{aligned}
				{v}_{r,\ell}^{s-1}(\xi,x)\ll M_{r,\ell}(s-1) \frac{|\xi|^{\ell-k_{s-1}}
				e^{c|\xi|^{\kappa_{s-1}}}}{\Gamma(\ell/k_{s-1})}
					\frac{\Theta^{(\ell)}(X)}{\ell !},
				\end{aligned} \label{Vs-2}
			\end{equation}
			where $\sum_{\ell=1}^{\infty}M_{r,\ell}(s-1)\tau^{\ell}$ converges in a neighborhood of $\tau=0$.
			\par
			 Let 
      \begin{equation}\left \{
 	     	\begin{aligned}
	      	V^{s}(\xi,x,\varepsilon):=&(\mathscr{A}_{k_{s}.k_{s-1},\theta_{s-1}}V^{s-1})(\xi,x,\varepsilon). \\
 	     	{v}_{r,\ell}^{s}(\xi,x):
 	     	=& \mathscr{A}_{k_{s}.k_{s-1},\theta_{s-1}}{v}_{r,\ell}^{s-1}(\xi,x),.
 	     	\end{aligned}\right . \label{Vs}
      \end{equation} 
 	     It follows from $|\theta_s-\theta_{s-1}|\leq \pi/2\kappa_{s-1}$ and  \eqref{Vs-2} that 
 	     ${v}_{r,\ell}^{s}(\xi,x)$ is holomorphic in ${S_s^*}_0$ with respect to $\xi$.
	        Our aim is to show $V^{s}(\xi,x,\varepsilon)$ has holomorphic prolongation to 
	        an infinite sector $S_s^*$ with desired estimate. 
		      	We have from ${g}_{r,\alpha,A}^{s}(\xi,x) 
			=(\mathscr{A}_{k_{s},k_{s-1},\theta_{s-1}}	{g}_{r,\alpha,A}^{s-1})(\xi,x) $,  
		\begin{equation}
			\begin{aligned}\ & \mathscr{A}_{k_{s},k_{s-1},\theta_{s-1}}\big(
					{g}_{r,\alpha,A}^{s-1}(\xi,x)\underset{k_{s-1}}{*}(V^{s-1})^{{*} \alpha}
					\underset{k_{s-1}}{*}  (\nabla _x  V^{s-1})^{* A} \big)\\
				&	={g}_{r,\alpha,A}^{s}(\xi,x)\underset{k_{s}}{*}
					(\mathscr{A}_{k_{s},k_{s-1},\theta_{s-1}}V^{s-1})^{{*} \alpha}
					\underset{k_{s}}{*}  
					(\mathscr{A}_{k_{s},k_{s-1},\theta_{s-1}}\nabla _xV^{s-1})^{* A}\\
				&	={g}_{r,\alpha,A}^{s}(\xi,x)\underset{k_{s}}{*}
					(V^{s})^{{*} \alpha}
					\underset{k_{s}}{*}(\nabla _xV^{s})^{* A}.
		\end{aligned} \label{rel}
	\end{equation}  
			and 	$\mathcal{G}_{r,\ell}(s)(\xi,x)=
	    \big(\mathscr{A}_{k_{s},k_{s-1},\theta_{s-1}}\mathcal{G}_{r,\ell}(s-1)\big)(\xi,x)$.
	    This means $V^{s}(\xi,x,\varepsilon)=(\mathscr{A}_{k_{s}.k_{s-1},\theta_{s-1}}
			V^{s-1})(\xi,x,\varepsilon)$ satisfies $(C_{s,\varepsilon})$ in ${S^*_s}_0$.  
			 Let ${v}_r^s(\xi,x,\varepsilon)=\sum_{\ell=1}^{\infty}{v}_{r,\ell}^s(\xi,x)\varepsilon^{\ell}$. 
		 Then we have for $\xi \in {S^*_s}_0$  
		  \begin{equation}\left \{
		 	\begin{aligned}\
			 	& (\xi \partial_{\xi}+k_s+1){v}_{r,1}^s(\xi,x)=g_{r,0,0}^{s}(\xi,x)\\
			 	& (\xi \partial_{\xi}+k_s+1){v}_{r,\ell}^s(\xi,x)=\mathcal{G}_{r,\ell}(s)(\xi,x)\quad \ell 
			 	\geq 2.
		 	\end{aligned}\right .\label{vs}
	 	\end{equation}
  		 	We have to show that $\{v_{r,\ell}^s(\xi,x)\}$ are holomorphically extensible to $S^*_s$ and 
			${v}_{r,\ell}^s(\xi,x) \in {\rm Exp}(\kappa_{s}, S(s)\times \Omega_0)$. For this purpose we 
			study \eqref{vs} that $\{{v}_{r,\ell}^s(\xi,x)\}_{r=1}^m$ satisfy. 
			Since $\{g_r(t,x,U,P)\}_{i=1}^m$ are ${\bf k}$-multisummable, \eqref{gs-est} holds.
			We have 
		\begin{prop} {\rm (1)}  
				${v}_{r,\ell}^s(\xi,x)$ has  holomorphic prolongation to infinite sector $S^*_s$ and 
				${v}_{r,\ell}^s(\xi,x)	\in {\rm Exp}(\kappa_s,S^*_{s}\times\Omega_0)$. There are 
				constants $\{M_{r,\ell}(s)\}_{r=1}^m$\; $(\ell\geq 1)$ and $c$ such that 
			\begin{equation}
				\begin{aligned}
					{v}_{r,\ell}^s(\xi,x)\ll M_{r,\ell}(s)
					 \frac{|\xi|^{\ell-k_{s}}e^{c|\xi|^{\kappa_{s}}}}{\Gamma(\ell/k_{s})}
					\frac{\Theta^{(\ell)}(X)}{\ell !}
				\end{aligned}
			\end{equation}
				and $\sum_{\ell=1}^{\infty}M_{r,\ell}(s)\tau^{\ell}$ converges in a neighborhood of $\tau=0$. \\
			{\rm (2)}.
			  Let $\varepsilon=1$ and ${v}_r^s(\xi,x)=\sum_{\ell=1}^{\infty}{v}_{r,\ell}^s(\xi,x)$.
				Then $\{{v}_r^s(\xi,x);1\leq r\leq m\}$ converge in a neighborhood of $x=0$
				and 
				\begin{equation}
					\begin{aligned}
					|{v}_r^s(\xi,x|\leq M|\xi|^{1-k_s}e^{c'|\xi|^{\kappa_s}}
					\end{aligned}
				\end{equation}
					${V}^s(\xi,x)=({v}_{1}^s(\xi,x),\dots,{v}_{m}^s(\xi,x))$ satisfies {\rm(}$C_{s}${\rm)}. 
				 \label{prop-s}
		\end{prop}
			  Proof of Proposition \ref{prop-s} is almost same as Borel summable case. 
			  We can show it in  repeatable way.
			  Consequently we get ${V}^s(\xi,x)$ from ${V}^{s-1}(\xi,x)$	and 
			  $\{V^s(\xi,x)\}_{s=1}^{p}$ successively, 
			\begin{equation}\left \{
				\begin{aligned}
		   	 &	 V^1(\xi,x)=W(\xi,x) \\				   
		   	 &	 {V}^s(\xi,x)= (\mathscr{A}_{k_s,k_{s-1},\theta_{s-1}}
		   	 		{V}^{s-1})(\xi,x) \;\; 2\leq s \leq p .
				\end{aligned}\right .
			\end{equation}
				$W(\xi,x) $ is a solution of $(C_{1})$ determined in Proposition \ref{prop-1}.  Let  
			\begin{equation}
				\begin{aligned}
				&   V(t,x)= (\mathscr{L}_{k_p,\theta_{p}}{V}^{p})(t,x)
				\end{aligned}
			\end{equation}
				Then ${\bf k}$-sum  $V(t,x)=({v}_{1}(\xi,x),\dots ,{v}_{m}(\xi,x) )$ is a 
				{\bf k}-multisummable solution of (CP-M) \eqref{CP-M} and we have Theorem  \ref{th-2}.
 	\section{\normalsize Appendix} 
 	         It remains to prove Lemmas \ref{lem2.2}, \ref{lem2.3} and \ref{lem5.1} and 
 	         Proposition \ref{prop5.3}. The set of 
 	         Borel summable functions is a subclass of multisummable functions. Hence 
 	         we show Lemma \ref{lem5.1} and Proposition \ref{prop5.3}. 	         
	 	       Let $ \Omega_0=\{x\in \mathbb{C}^n; |x|<R_0\}$,  $\Omega'=\{Z\in\mathbb{C}^N; |Z|<R\}$ and
	 	       $\Omega=\Omega_0\times \Omega'$. 
	 	       Let $\widetilde{f}(t,x,Z) \in \mathscr{O}(\Omega)[[t]]$ be ${\bf k}=(k_1,k_2,\cdots,k_p)$-
	 	       multisummable in a multidirection
	 	       ${\bf \theta}=(\theta_1,\theta_2,\cdots,\theta_p)$ and 
	 	       $S_i^*=\{ \xi; |\arg \xi-\theta_i|<\varepsilon_i\}$.
	 	       Assume $\widetilde{f}(0,x,Z)=0$. 
	 	       Then $f^1(\xi,x,Z)=(\widetilde{\mathscr{B}}_{k_1}\widetilde{f})(\xi,x,Z)$
	 	       and $f^s(\xi,x,Y)=(\mathscr{A}_{k_s,k_{s-1}\theta_{s-1}}f^{s-1})(\xi,x,Y)$. It holds that
	 	       \begin{equation}
	 	       	\begin{aligned}
	 	       	|f^{s}(\xi,x,Z)|\leq K |\xi|^{1-k_s}e^{c|\xi|^{\kappa_s}}  \mbox{\it in $S_s^*\times \Omega$}. 
	 	       	\end{aligned}
	 	       \end{equation}
           By expanding $f^{s}(\xi,x,Z)$ at $Z=0$, we have
			     \begin{equation}\left \{
			  		\begin{aligned}
			  		    &  f^{s}(\xi,x,Z)=\sum_{C \in\mathbb{N}^N }f^{s}_{C}(\xi,x)Z^{C}, \quad C=(C_1,\cdots,C_N)
						     \in \mathbb{N}^N \\
			  		    & |f^{s}_{C}(\xi,x)|\leq \frac{ K}{R^{|C|}} |\xi|^{1-k_s}e^{c|\xi|^{\kappa_s}}.
			  		\end{aligned}\right .\label{6.2}
			  	\end{equation}
	 		\begin{lem} \label{lem6-1}
		 		  Let $\widetilde{f}(t,x,Z) \in \mathscr{O}(\Omega)[[t]]$ be 
		 		${\bf k}=(k_1,k_2,\cdots,k_p)$-multisummable in a multidirection ${\mathbf \theta}$ with
				$\widetilde{f}(0,x,Z)=0$.   Let 
				 		  \begin{equation*}
				 		  	\begin{aligned}
				 		  	& {h}_{C}(\xi,x)={f}^s_{C}(\xi,x)\underset{k_s}{*}\frac{\xi^{|C|-k_s}}{\Gamma(|C|/k_s)}
				 		  	\;\;\;
				 		  		 C\not=0,\quad {h}_0(\xi,x)={f}^s_{0}(\xi,x) \\
				 			 & {h}(\xi,x,Z)=\sum_{C\in \mathbb{N}^N}{h}_{C}(\xi,x)Z^C.
				 		  	\end{aligned}
				 		  \end{equation*}	
				 		 Then there exist $M$ and $c'$ such that for $|Z|<R'<R$
					\begin{equation}
							\begin{aligned}  
								& |{h}(\xi,x,Z)|\leq  \frac{M|\xi|^{1-k_s}e^{c'|\xi|^{\kappa_s}}}{\Gamma(1/k_s)}
								\quad (\xi,x)\in S^*_s\times \Omega_0.				
							\end{aligned}\label{6.3}
					\end{equation} 
				   \end{lem} 
				   \begin{proof}
				      We have from \eqref{6.2} for $C\not=0$
				    \begin{equation}
							\begin{aligned}  &
							|{h}_C(\xi,x)|=|{f}_{C}^s(\xi,x)\underset{k_s}{*}
							\frac{\xi^{|C|-k_s}}{\Gamma(|C|/k_s)}|\leq 
							\frac{K|\xi|^{|C|+1-k_s}e^{c|\xi|^{\kappa_s}}}{R^{|C|}\Gamma(\frac{|C|+1}{k_s})}.
							\end{aligned}\label{hC}
						\end{equation}
							Hence if $|Z|<R'<R$,  there exist $M$ and $c'$ such that 
						\begin{equation*}
							\begin{aligned} & |{h}(\xi,x,Z)|\leq
							   			\sum_{C \in \mathbb{N}^N}|{h}_{C}(\xi,x)Z^{C}| \\
							 \leq & Ke^{c|\xi|^{\kappa_s}}\sum_{C \in \mathbb{N}^N}
							(\frac{R'}{R})^{|C|} \frac{|\xi|^{|C|+1-k_s}}{\Gamma(\frac{|C|+1}{k_s})}
							\leq \frac{M|\xi|^{1-k_s}e^{c'|\xi|^{\kappa_s}}}{\Gamma(1/k_s)}.
							\end{aligned}
						\end{equation*}	
				   \end{proof}  	
      	 \begin{lem} 
	 		   	  	Let $\widetilde{g}(t,x,Z)=\widetilde{f}(t,x,tZ)$. Then
						  $\widetilde{g}(t,x,Z)$ is $\bf k$-multisummable. \label{lem6.2}
				\end{lem}
					\begin{proof} 
				  	We have  $\widetilde{g}(t,x,Z)=\sum_{C\in \mathbb{N}^N}\widetilde{f}_{C}(t,x)t^{|C|}Z^{C}$.
					  First assume $\widetilde{f}(0,x,Z)=0$. Then
					  $\widehat{g}(\xi,x,Z)=\widehat{f}_{0}(\xi,x)+
				    	\sum_{C\in \mathbb{N}^N,C\not=0}\widehat{f}_{C}(\xi,x)\underset{k_1}{*}
						    \frac{\xi^{|C|-k_1}}{\Gamma(|C|/k_1)}Z^{C}$ and 
				    \begin{equation*}
				    	\begin{aligned} 
				    &	{g}^s(\xi,x,Z)={f}_{0}^s(\xi,x)+
					    \sum_{C\in \mathbb{N}^N, C\not=0}{f}_{C}^s(\xi,x)\underset{k_s}{*}
					    \frac{\xi^{|C|-k_s}}{\Gamma(|C|/k_s)}Z^{C}.
				  	\end{aligned}
			   \end{equation*}
				    We have from Lemma \ref{lem6-1}  
					 \begin{equation}
								\begin{aligned}  
								&  |{g}^s(\xi,x,Z)|\leq  \frac{M|\xi|^{1-k_s}e^{c'|\xi|^{\kappa_s}}}{\Gamma(1/k_s)}
								\quad (\xi,x)\in S^*_s\times \Omega_0, |Z|<R'.
								\end{aligned}
					\end{equation} 
					and  $\widetilde{g}(t,x,Z)$ is $\bf k$-multisummable.
						Next assume $\widetilde{f}(0,x,Z)\not=0$. 
						Let $f^*(t,x,Z)=f(t,x,Z)-f(0,x,Z)$. Then $f^*(0,x,Z)=0$ and $f^*(t,x,tZ)$ is 
					$\bf k$-multisummable. $f(0,x,tZ)$ is holomorphic at $t=0$, hence 
					$\widetilde{g}(t,x,Z)=\widetilde{f}^*(t,x,tZ)+\widetilde{f}(0,x,tZ)$ is also 
					$\bf k$-multisummable. 
				\end{proof}
			  Thus we get Lemma \ref{lem5.1}. We proceed to show Proposition \ref{prop5.3}. We note 
					\begin{equation}\left \{
						\begin{aligned}
						& g_i(t,x, V, P):=f_i(t,x, tV, tP),\;\; g_i(0,x,V,P)=f_i(0,x,0,0)=0 \\
						&  g_i(t,x, V, P)=\sum_{\alpha,A}g_{i,\alpha,A}(t,x)V^{\alpha}P^A . 
						\end{aligned}\right .
					\end{equation}
					and
						\begin{equation}
						\begin{aligned}
						&  g_i^s(\xi,x, V, P)=\sum_{\alpha,A}g_{i,\alpha,A}^{s}(\xi,x)V^{\alpha}P^A. 
						\end{aligned}
					\end{equation}
					Let $C=(\alpha,A)$ and $h_{C}(\xi,x)=g_{i,\alpha,A}^{s}(\xi,x)$ in \eqref{hC} 
					in lemma  \ref{lem6-1}. Then  
						\begin{equation*}\left \{
							\begin{aligned} & 
							|g_{i,\alpha,A}^{s}(\xi,x)| \leq \frac{K|\xi|^{|\alpha|+|A|+1-k_s}e^{c|\xi|^{\kappa_s}}}
							{R^{|\alpha|+|A|}\Gamma(\frac{|\alpha|+|A|+1}{k_s})}, \\
							&|{g}_i^s(\xi,x,Z)|\leq  \frac{M|\xi|^{1-k_s}e^{c|\xi|^{\kappa_s}}}{\Gamma(1/k_s)}			
							\end{aligned}\right .
						\end{equation*}
				and
						\begin{equation}
							\begin{aligned} & 
							g_{i,\alpha,A}^{s}(\xi,x)\ll \frac{M'|\xi|^{|\alpha|+|A|+1-k_s}e^{c|\xi|^{\kappa_s}}}
							{R^{
							|\alpha|+|A|}\Gamma(\frac{|\alpha|+|A|+1}{k_s})}{\Theta(X)}. 
							\end{aligned}
						\end{equation}
			       Hence  it follows from Lemma 3.2 that ${\Theta(X)} \ll \frac{B^{|\alpha|+|A|}
				\Theta^{|\alpha|+|A|}(X)}{(|\alpha|+|A|)!}$ and
				\begin{equation*}
					\begin{aligned}& 
					\frac{M'|\xi|^{|\alpha|+|A|+1-k_s}e^{c|\xi|^{\kappa_s}}}
							{R^{||\alpha|+|A|}\Gamma(\frac{|\alpha|+|A|+1}{k_s})}{\Theta(X)} , \\
					&	\ll \frac{M'|\xi|^{|\alpha|+|A|+1-k_s}e^{c|\xi|^{\kappa_s}}}
							{R^{|\alpha|+|A|}}
							 \frac{B^{|\alpha|+|A|}}{(|\alpha|+|A|)!} \Theta^{|\alpha|+|A|}(X).
						\end{aligned}
				\end{equation*}
					There exist positive constants 
					$G_{i,\alpha,A}(s)= M'(\frac{B}{R})^{|\alpha|+|A|}$ and $c'>c$  such that
				\begin{equation}\left \{
					\begin{aligned}
					\ & g_{i,0,0}^{s}(\xi,x)\ll G_{i,0,0}(s)\frac{\xi^{1-k_{s}}e^{c'|\xi|^{\kappa_{s}}}}
							{\Gamma(\frac{1}{k_{s}})}{\Theta(X)} \\
					 &	g_{i,\alpha,A}^{s}(\xi,x)\ll G_{i,\alpha,A}(s)\frac{\xi^{|\alpha|+|A|-k_{s}}
					 e^{c'|\xi|^{\kappa_{s}}}}
							{\Gamma(\frac{|\alpha|+|A|}{k_{s}})}\frac{\Theta^{|\alpha|+|A|}(X)}{(|\alpha|+|A|)!}
							\quad |\alpha|+|A|\geq 1
					\end{aligned}\right .
				\end{equation}
					and $ \sum_{\alpha,A}  G_{i,\alpha,A}(s)U^{\alpha}p^A$ converges in 
					a neighborhood of $(U,p)=(0,0)$. Thus we get Proposition \ref{prop5.3}.
		
\end{document}